\documentclass{article}
\usepackage{hyperref}
\usepackage{graphicx}
\usepackage[all]{xy} 
\usepackage{amsmath}
\usepackage{amsthm}
\usepackage{array}  
\usepackage{amssymb}
\usepackage{calc}
\usepackage{supertabular}
\usepackage{longtable}
\usepackage{txfonts}
\usepackage{textcomp}
\usepackage{colortbl}
\usepackage{tabularx} 
\usepackage[table]{xcolor}
\usepackage{lscape}
\usepackage{rotating}
\usepackage{tensor}
\usepackage{arydshln} 
\usepackage{bbm} 
\usepackage{appendix} 
\usepackage{placeins} 
\usepackage[T1]{fontenc} 
\usepackage{titletoc}
\usepackage{pgf}
\usepackage{tikz}
\usetikzlibrary{matrix,arrows}
\usetikzlibrary{patterns}
\usetikzlibrary{shapes,decorations,shadows}
\DeclareMathOperator{\Ext}{Ext}

\DeclareMathOperator{\Hom}{Hom} 
\DeclareMathOperator{\rad}{rad} 
\DeclareMathOperator{\inj}{inj}

\DeclareMathOperator{\id}{id}

\DeclareMathOperator{\rep}{rep} 
\DeclareMathOperator{\olp}{olp} 
\DeclareMathOperator{\eolp}{eolp} 
\DeclareMathOperator{\GL}{GL} 
\DeclareMathOperator{\Iso}{Iso} 
\DeclareMathOperator{\mdeg}{mdeg} 
 
\DeclareMathOperator{\codim}{codim} 
 
\definecolor{lblue}{rgb}{0.3,0.0,4.4}
\newcommand{\ef}{\underline{e}}
\newcommand{\eef}{\underline{e}'}  
\newcommand*{\punkte}{\dots\unkern}
\newcolumntype{C}[1]{>{\centering\arraybackslash}p{#1}}
\newcommand{\A}{\mathcal{A}} 
\newcommand{\Fa}{\mathcal{F}}

\newcommand{\U}{\mathcal{U}} 
\newcommand{\Q}{\mathcal{Q}} 
\newcommand{\Orb}{\mathcal{O}} 
\newcommand{\N}{\mathcal{N}} 
\newcommand{\V}{\mathcal{V}} 

\newcommand{\Sa}{\mathcal{S}}
\newcommand{\dimv}{\underline{\dim}}
\newcommand{\dfp}{\underline{d}_{P}} 
\newcommand{\df}{\underline{d}} 
\newcommand{\dfs}{\underline{d}_B}

\newtheorem{theorem}{Theorem}[section]
\newtheorem{lemma}[theorem]{Lemma}

\newtheorem{proposition}[theorem]{Proposition}
\newtheorem{corollary}[theorem]{Corollary}

\newtheorem{example}[theorem]{Example}

\begin{document}
\parindent0pt
\title{\bf Finite parabolic conjugation on varieties of nilpotent matrices}

\author{Magdalena Boos\\ Fachbereich C - Mathematik\\ Bergische Universit\"at Wuppertal\\ D - 42097 Wuppertal\\ boos@math.uni-wuppertal.de}
\date{}
\maketitle

\begin{abstract}
We consider the conjugation-action of an arbitrary upper-block parabolic subgroup of $\GL_n(\mathbf{C})$ on the variety of $x$-nilpotent complex matrices and translate it to a representation-theoretic context.  We obtain a criterion as to whether the action admits a finite number of orbits and specify a system of representatives for the orbits in the finite case of $2$-nilpotent matrices. Furthermore, we give a set-theoretic description of their closures and specify the minimal degenerations in detail for the action of the Borel subgroup. We show that in all non-finite cases, the corresponding quiver algebra is of wild representation type.

\end{abstract}
\section{Introduction}\label{intro}
In algebraic Lie theory, algebraic group actions on affine varieties can be studied "vertically", that is, by their orbits and their closures.\\[1ex]
A well-known example is the study of the adjoint action of a reductive algebraic group on its Lie algebra and numerous variants thereof, in particular the conjugacy classes of complex (nilpotent) square matrices.
In 1870, the classification of these by so-called Jordan normal forms was described by M. Jordan \cite{Jo1,Jo2}. Their closures were described by M. Gerstenhaber \cite{Ger} and W. Hesselink \cite{He} in the second half of the twentieth century in terms of partitions and visualized by combinatorial objects named Young Diagrams.\\[1ex]
We turn our main attention towards algebraic non-reductive group actions that are induced by the conjugation action of the general linear group $\GL_n$ over $\mathbf{C}$. For example, the standard parabolic subgroups $P$ (and, therefore, the standard Borel subgroup $B$) and the unipotent subgroup $U$ of $\GL_n$ are not reductive. 
It suggests itself to consider their action on the variety $\N_n^{(x)}$ of $x$-nilpotent matrices of square size $n$ via conjugation which we discuss in this work.\\[1ex]
A recent development in this field is A. Melnikov's study of the $B$-action on the variety of upper-triangular $2$-nilpotent matrices via conjugation \cite{Me1,Me2} motivated by Springer Theory. The detailed description of the orbits and their closures is given in terms of so-called link patterns; these are combinatorial objects visualizing the set of involutions in the symmetric group $S_n$.  In \cite{BoRe}, M. Reineke and the author generalize these results to the Borel-orbits of all $2$-nilpotent matrices and describe the minimal, disjoint degenerations corresponding to their orbit closure relations. Furthermore, L. Fresse describes singularities in the upper-triangular orbit closures by translating the group action to a certain group action on Springer fibres (see \cite{Fr}).\\[1ex]
Another recent outcome is L. Hille's and G. R\"ohrle's study of the action of $P$ on its unipotent radical $P_u$, and on the corresponding Lie algebra $\mathfrak{p}_u$ (see \cite{HiRoe}). They obtain a criterion which varifies that the number of orbits is finite if and only if the nilpotency class of $P_u$ is less or equal than $4$. This result is generalized to all classical groups $G$. Also, Magmar, Weyman and Zelevinski discuss flag varieties of finite types in \cite{MWZ}. Given a semi-simple Lie algebra $\mathfrak{g}$ and its Lie group $G$, D. Panyushev considers the adjoint action in \cite{Pan1} and shows that, given a nilpotent element $e\in\mathfrak{g}\backslash\{0\}$, the orbit $G.e$ is spherical if and only if $({\rm ad}_e)^4=0$.  The notion of sphericity translates to $G.e$ admitting only a finite number of Borel-orbits, see \cite{Br1, Vin}.\\[1ex]
In this work, we make use of a translation of the classification problem of the $P$-orbits in $\N_n^{(x)}$ to  the description of certain isomorphism classes of representations of a finite-dimensional algebra in Section \ref{transsect}. By making use of this translation, we describe the $P$-orbits in $\N_n^{(2)}$ in Section \ref{x2} as well as their closures in Section \ref{closures} in detail. Here, all minimal degenerations for the Borel-action are specified as well. This particular action admits only a finite number of orbits and we describe the finite case of a maximal parabolic acting on $3$-nilpotent matrices in Section \ref{degree3}. We find a criterion as to whether the action admits a finite number of orbits in Section \ref{fincrit} and show that in every remaining case, the corresponding quiver algebra is of wild representation type. \\[1ex]
The results stated in this article represent a part of the outcome of the dissertation \cite{B1}.\\[1ex]
{\bf Acknowledgments:} The author would like to thank M. Reineke for various valuable discussions concerning the methods and results of this work. Furthermore,  A. Melnikov, K. Bongartz and M. Bender are being thanked for inspirational thoughts and helpful remarks.\\
The published version of this article is \cite{B2}.

\section{Theoretical background}\label{theory}
We denote by $K\coloneqq \mathbf{C}$ the field of complex numbers and by $\GL_n\coloneqq\GL_n(K)$ the general linear group for a fixed integer $n\in\textbf{N}$ regarded as an affine variety. We include basic knowledge about the representation theory of finite-dimensional algebras \cite{ASS}.\\[1ex]
A \textit{finite quiver} $\Q$ is a directed graph $\Q=(\Q_0,\Q_1,s,t)$, such that $\Q_0$ is a finite set of \textit{vertices} and $\Q_1$ is a finite set of \textit{arrows}, whose elements are written as $\alpha\colon s(\alpha)\rightarrow t(\alpha)$.
The \textit{path algebra} $K\Q$ is defined as the $K$-vector space with a basis consisting of all paths in $\Q$, that is, sequences of arrows $\omega=\alpha_s\punkte\alpha_1$, such that $t(\alpha_{k})=s(\alpha_{k+1})$ for all $k\in\{1,\punkte,s-1\}$; formally included is a path $\varepsilon_i$ of length zero for each $i\in \Q_0$ starting and ending in $i$. The multiplication  is defined by
\begin{center}
 $\omega\cdot\omega'=\left\{\begin{array}{ll}\omega\omega',&~\textrm{if}~t(\beta_t)=s(\alpha_1);\\
0,&~\textrm{otherwise.}\end{array}\right.$\end{center}
where $\omega\omega'$ is the  concatenation of paths $\omega= \alpha_s ... \alpha_1$ and $\omega' = \beta_t ... \beta_1$.\\[1ex]
We define the \textit{radical} $\rad(K\Q)$ of $K\Q$ to be the (two-sided) ideal generated by all paths of positive length; then an arbitrary ideal $I$ of $K\Q$ is called \textit{admissible} if there exists an integer $s$ with $\rad(K\Q)^s\subset I\subset\rad(K\Q)^2$.\\[1ex]
A finite-dimensional \textit{$K$-representation} of $\Q$ is a tuple \[((M_i)_{i\in \Q_0},(M_\alpha\colon M_i\rightarrow M_j)_{(\alpha\colon i\rightarrow j)\in \Q_1}),\] where the $M_i$ are $K$-vector spaces, and the $M_{\alpha}$ are $K$-linear maps.\\[1ex]
 A \textit{morphism of representations} $M=((M_i)_{i\in \Q_0},(M_\alpha)_{\alpha\in \Q_1})$ and
 \mbox{$M'=((M'_i)_{i\in \Q_0},(M'_\alpha)_{\alpha\in \Q_1})$} consists of a tuple of $K$-linear maps $(f_i\colon M_i\rightarrow M'_i)_{i\in \Q_0}$, such that $f_jM_\alpha=M'_\alpha f_i$ for every arrow $\alpha\colon i\rightarrow j$ in $\Q_1$.\\[1ex]
For a representation $M$ and a path $\omega$ in $\Q$ as above, we denote $M_\omega=M_{\alpha_s}\cdot\punkte\cdot M_{\alpha_1}$. A representation $M$ is called \textit{bound by $I$} if $\sum_\omega\lambda_\omega M_\omega=0$ whenever $\sum_\omega\lambda_\omega\omega\in I$.\\[1ex]
These definitions yield certain categories as follows: We denote by $\rep_K(\Q)$ the abelian $K$-linear category of all representations of $\Q$ and by   $\rep_K(\Q,I)$ the category of representations of $\Q$ bound by $I$; the latter is equivalent to the category of finite-dimensional $K\Q/I$-representations.\\[1ex]
Given a representation $M$ of $\Q$, its \textit{dimension vector} $\dimv M\in\mathbf{N}\Q_0$ is defined by $(\dimv M)_{i}=\dim_k M_i$ for $i\in \Q_0$. Let us fix a dimension vector $\df\in\mathbf{N}\Q_0$, then we denote by $\rep_K(\Q,I)(\df)$ the full subcategory of $\rep_K(\Q,I)$ which consists of representations of dimension vector $\df$.\\[1ex]
For certain classes of finite-dimensional algebras, a convenient tool for the classification of the indecomposable representations is the \textit{Auslander-Reiten quiver} $\Gamma(\Q,I)$ of $\rep_K(\Q,I)$. Its vertices $[M]$ are given by the isomorphism classes of indecomposable representations of $\rep_K(\Q,I)$; the arrows between two such vertices $[M]$ and $[M']$ are parametrized by a basis of the space of so-called irreducible maps $f\colon M\rightarrow M'$.\\[1ex] 
One standard technique to calculate the Auslander-Reiten quiver is the \textit{knitting process} (see, for example, \cite[IV.4]{ASS}).
 In some cases, the Auslander-Reiten quiver $\Gamma(\Q,I)$ can be calculated by using \textit{covering techniques} (see \cite{Ga3} or \cite{BoGa}). We will make use of the latter and describe some more details on these techniques later on.\\[1ex]
By defining the affine space $R_{\df}(\Q):= \bigoplus_{\alpha\colon i\rightarrow j}\Hom_K(K^{d_i},K^{d_j})$, one realizes that its points $m$ naturally correspond to representations $M\in\rep_K(\Q)(\df)$ with $M_i=K^{d_i}$ for $i\in \Q_0$. 
 Via this correspondence, the set of such representations bound by $I$ corresponds to a closed subvariety $R_{\df}(\Q,I)\subset R_{\df}(\Q)$.\\[1ex]
The algebraic group $\GL_{\df}=\prod_{i\in \Q_0}\GL_{d_i}$ acts on $R_{\df}(\Q)$ and on $R_{\df}(\Q,I)$ via base change, furthermore the $\GL_{\df}$-orbits $\Orb_M$ of this action are in bijection to the isomorphism classes of representations $M$ in $\rep_K(\Q,I)(\df)$.\\[1ex]
A finite-dimensional $K$-algebra $\A:=K\Q/I$, such that $\rep_K(\Q,I)$ is locally bounded is called of tame representation type (or simply \textit{tame}) if for every integer $d$ there is an integer $m_d$ and there are finitely generated $K[x]$-$\A$-bimodules $M_1,\punkte,M_{m_d}$ that are free over $K[x]$, such that for all but finitely many isomorphism classes of indecomposable right $\A$-modules $M$ of dimension $d$, there are elements $i\in\{1,\punkte,m\}$ and $\lambda\in K$, such that  $M\cong  K[x]/(x-\lambda)\otimes_{K[x]}M_i$.\\[1ex]
It is called of wild representation type (or simply \textit{wild}) if there is a finitely generated $K\langle X,Y\rangle$-$\A$-bimodule that is free over $K\langle X,Y\rangle$, such that the functor $\_\otimes_{K\langle X,Y\rangle}M $ sends non-isomorphic finite-dimensional $K\langle X,Y\rangle$-modules to non-isomorphic $\A$-modules.\\[1ex]
In 1979, J. A. Drozd proved the following theorem (see \cite{Dr}).
\begin{theorem}
 Every finite-dimensional algebra is either tame or wild.
\end{theorem}
The notion of a tame algebra $\A$ yields that there are at most $1$-parameter families of pairwise non-isomorphic indecomposable $\A$-modules; in the wild case there are parameter families of arbitrary many parameters of pairwise non-isomorphic indecomposable $\A$-modules. In order to show that an algebra is wild, it, thus, suffices to describe one particular such $2$-parameter family.\\[1ex]
The theorem of P. Gabriel (see \cite{Ga1}) shows that $K\Q$ is of finite type if and only if the underlying unoriented graph of $\Q$ is a disjoint union of Dynkin graphs $A$, $D$, $E_6$, $E_7$ or $E_8$.
The algebra $K\Q$ is representation-infinite and tame if and only if the underlying unoriented graph is a disjoint union of at least one extended Dynkin graph $\widetilde{A}$, $\widetilde{D}$, $\widetilde{E_6}$, $\widetilde{E_7}$ or $\widetilde{E_8}$ and Dynkin graphs.

 \section{Translation to a representation-theoretic setup}\label{transsect}
We fix a parabolic subgroup $P$ of $\GL_n$ of block sizes $(b_1,\punkte,b_p)$.\\[1ex] 
We define  $\Q_p$ to be the quiver
\begin{center}\small\begin{tikzpicture}
\matrix (m) [matrix of math nodes, row sep=0.01em,
column sep=1.5em, text height=0.5ex, text depth=0.1ex]
{\Q_p\colon & \bullet & \bullet &  \bullet & \cdots  & \bullet & \bullet  & \bullet \\ & \mathrm{1} & \mathrm{2} &  \mathrm{3} & &   \mathrm{p-2} &  \mathrm{p-1}  & \mathrm{p} \\ };
\path[->]
(m-1-2) edge node[above=0.05cm] {$\alpha_1$} (m-1-3)
(m-1-3) edge  node[above=0.05cm] {$\alpha_2$}(m-1-4)
(m-1-6) edge  node[above=0.05cm] {$\alpha_{p-2}$}(m-1-7)
(m-1-7) edge node[above=0.05cm] {$\alpha_{p-1}$} (m-1-8)
(m-1-8) edge [loop right] node{$\alpha$} (m-1-8);\end{tikzpicture}\end{center} 
and $\A(p,x)\coloneqq K \Q_p/I_x$  to be the finite-dimensional algebra, where $I_x\coloneqq (\alpha^x)$ is an admissible ideal. We fix the dimension vector 
\[\dfp\coloneqq(d_1,\punkte,d_p)\coloneqq(b_1,b_1+b_2, \punkte, b_1+...+b_p)\]
 and formally set $b_0=0$. As explained in Section \ref{theory}, the algebraic group $\GL_{\dfp}$ acts on $R_{\dfp}(\Q_p,I_x)$; the orbits of this action are in bijection with the isomorphism classes of representations in $\rep_{K}(\Q_p,I_x)(\dfp)$.\\[1ex]
Let us define $\rep_{K}^{\inj}(\Q_p,I_x)(\dfp)$ to be the full subcategory of $\rep_{K}(\Q_p,I_x)(\dfp)$ consisting of representations $((M_i)_{1\leq i\leq p},(M_{\rho})_{\rho\in \Q_1})$, such that $M_{\rho}$ is injective if $\rho=\alpha_i$ for every $i\in\{1,\punkte, p-1\}$. Corresponding to this subcategory, there is an open subset $R_{\dfp}^{\inj}(\Q_p,I_x)\subset R_{\dfp}(\Q_p,I_x)$, which is stable under the $\GL_{\dfp}$-action.\\[1ex]
We denote $\Orb_M:=\GL_{\dfp}.m$ if $m\in R_{\dfp}^{\inj}(\Q_p,I_x)$ corresponds to $M\in\rep^{\inj}(\Q_p,I_x)(\dfp)$ as in Section \ref{theory}.
In order to describe the orbit closure $\overline{\Orb_M}$, we denote $M\leq_{\deg}M'$ if $\Orb_{M'}\subset\overline{\Orb_M}$ in $R_{\dfp}(\Q_p,I_x)$ for a representation $M'$ and say that $M'$ is a \textit{degeneration} of $M$. 
 Of course, in order to describe all degenerations, it is sufficient to calculate all \textit{minimal degenerations} $M<_{\mdeg}M'$, that is, degenerations $M<_{\deg}M'$, such that if $M\leq_{\deg}L\leq_{\deg}M'$, then $M\cong L$ or $M'\cong L$.
The following lemma is a slightly generalized version of \cite[Lemma 3.2]{BoRe}. The proof is similar, though.
\begin{lemma} \label{bijection}
There is an isomorphism $R_{\dfp}^{\inj}(\Q_p,I_x)\cong \GL_{\dfp}\times^{P}\N_n^{(x)}$. Thus, there exists a bijection $\Phi$ between the set of $P$-orbits in $\N_n^{(x)}$ and the set of $\GL_{\dfp}$-orbits in $R_{\dfp}^{\inj}(\Q_p,I_x)$, which sends an orbit $P.N\subseteq \N_n^{(x)}$ to the isomorphism class of the representation
\begin{center}\small\begin{tikzpicture}
\matrix (m) [matrix of math nodes, row sep=0.05em,
column sep=2em, text height=1.5ex, text depth=0.2ex]
{ K^{d_1} & K^{d_2} & K^{d_3} & \cdots  & K^{d_{p-2}} & K^{d_{p-1}}  & K^{n}\\ };
\path[->]
(m-1-1) edge node[above=0.05cm] {$\epsilon_1$} (m-1-2)
(m-1-2) edge  node[above=0.05cm] {$\epsilon_2$}(m-1-3)
(m-1-3) edge  (m-1-4)
(m-1-4) edge  (m-1-5)
(m-1-5) edge  node[above=0.05cm] {$\epsilon_{p-2}$}(m-1-6)
(m-1-6) edge node[above=0.05cm] {$\epsilon_{p-1}$} (m-1-7)
(m-1-7) edge [loop right] node{$N$} (m-1-7);\end{tikzpicture}\end{center}
 (denoted $M^N$) with natural embeddings $\epsilon_i\colon K^{d_i}\hookrightarrow K^{d_{i+1}}$. This bijection preserves orbit closure relations,
 dimensions of stabilizers (of single points) and codimensions.
\end{lemma}
Due to considerations of different parabolic subgroups and nilpotency degrees, the classification of the corresponding isomorphism classes of representations differs wildly.

\section[P-orbits in the variety of 2-nilpotent matrices]{$P$-orbits in $\N_n^{(2)}$}\label{x2}
Let us consider the action of $P$ on the variety $\N_n^{(2)}$ of $2$-nilpotent $n\times n$- matrices. \\[1ex]
As the theorem of W. Krull, R. Remak and O. Schmidt states, every representation in $\rep_{K}(\Q_p,I_2)$ can be decomposed into a direct sum of indecomposables, which is unique up to permutations and isomorphisms. Following \cite[Theorem 3.3]{BoRe}, the following lemma classifies the indecomposables in $\rep_{K}^{\inj}(\Q_p,I_2)$.
\begin{lemma}\label{indec2nilp}
 Up to isomorphisms, the indecomposable representations in $\rep_{K}^{\inj}(\Q_p,I_2)$ are for $1\leq i,j\leq p$ and (1) $j\leq i$ or (2)  $j> i$\\[1ex]
$\begin{tikzpicture}
\matrix (m) [matrix of math nodes, row sep=0.02em,
column sep=0.08em, text height=1.0ex, text depth=0.25ex]
{(1)~\U_{i,j}:& 0 & \xrightarrow{0} & \cdots &  0 & \xrightarrow{0} & K & \xrightarrow{id} & \cdots  & K & \xrightarrow{e_1} & K^{2} & \xrightarrow{id}& \cdots & \xrightarrow{id} & K^{2}  \\};
\path[->]
(m-1-16) edge [loop right] node{$\alpha$} (m-1-16);
\end{tikzpicture}$\\[1ex]
$\begin{tikzpicture}
\matrix (m) [matrix of math nodes, row sep=0.02em,
column sep=0.08em, text height=1.0ex, text depth=0.25ex]
{(2)~\U_{i,j}:& 0 & \xrightarrow{0} & \cdots  & 0 & \xrightarrow{0} & K & \xrightarrow{id} & \cdots  & K & \xrightarrow{e_2} & K^{2} & \xrightarrow{id}& \cdots & \xrightarrow{id} & K^{2}  \\ };
\path[->]
(m-1-16) edge [loop right] node{$\alpha$} (m-1-16);
\end{tikzpicture}$\\[1ex]
$\begin{tikzpicture}
\matrix (m) [matrix of math nodes, row sep=0.02em,
column sep=0.1em, text height=1.0ex, text depth=0.25ex]
{\V_{i}:&  0 & \xrightarrow{0} & \cdots & \xrightarrow{0} & 0 & \xrightarrow{0} & K & \xrightarrow{id} & \cdots & \xrightarrow{id} & K  \\};
\path[->]
(m-1-12) edge [loop right] node{$0$} (m-1-12);
\end{tikzpicture}$\\[1ex]
Here, $e_1$ and $e_2$ are the standard coordinate vectors of $K^2$ and $\alpha\cdot e_1=e_2$, $\alpha\cdot e_2=0$.
\end{lemma}
 An \textit{enhanced oriented link pattern} of type $(b_1,\punkte,b_p)$ is an oriented graph on the vertices $\{1,\punkte ,p\}$ together with a (possibly empty) set of of dots at each vertex, such that the sum of the numbers of sources, targets and dots at every vertex $i$ equals $b_i$. Clearly, an enhanced oriented link pattern of a fixed type is far from being unique. We denote an enhanced oriented link pattern as a sequence of tuples $(i_1,j_1)...(i_k,j_k)$, such that there are arrows $j_x \rightarrow i_x$ for alle $x$ and $j_1\leq j_2 \leq ...$.\\[1ex]
For example, an enhanced oriented link pattern of type $(3,2,6,2,6)$ is given by
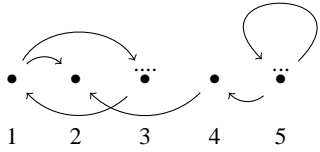
\begin{center}{\small\begin{tikzpicture}[descr/.style={fill=white,inner sep=2.5pt}]
  \matrix (m) [matrix of math nodes, row sep=1.01em, column sep=1.5em, text height=1.5ex, text depth=0.25ex]
 {\bullet & \bullet & \ddddot{\bullet} & \bullet & \dddot{\bullet} \\
1 & 2 & 3 & 4 &5\\};
\path[->,font=\scriptsize]
(m-1-1) edge [bend left=60]  (m-1-3)
(m-1-3) edge [bend left=45]   (m-1-1)
(m-1-4) edge [bend left=45]   (m-1-2)
(m-1-1) edge [bend left=45]   (m-1-2)
(m-1-5) edge [loop]   (m-1-5)
(m-1-5) edge [bend left=45]   (m-1-4);
  \end{tikzpicture}.}\end{center} 

\begin{theorem}\label{paraboliccase}
There are natural bijections between 
\begin{enumerate}
\item $P$-orbits in $\N_n^{(2)}$,
\item  isomorphism classes in $\rep^{\inj}_{K}(\Q_p,I_2)$ of dimension vector $\dfp$,
\item matrices $N=(p_{i,j})_{i,j}\in \mathbf{N}^{p\times p}$, such that
$\sum_j (p_{i,j}+ p_{j,i}) \leq b_i$ for all $i\in\{1,\punkte,p\}$,
\item   and enhanced oriented link patterns of type $(b_1,\punkte,b_p)$.
\end{enumerate}
Moreover, if the isomorphism class of $M$ corresponds to a matrix $N$ under this bijection, the orbit $\Orb_M\subset R_{\dfp}^{\inj}(\Q_p,I_2)$ and the orbit $P.N\subset\N_n^{(2)}$ correspond to each other via the bijection $\Phi$ of Lemma \ref{bijection}.
\end{theorem}
The proof is similar to the proof of \cite[Theorem 3.4]{BoRe}.
Note that the multiplicity of the indecomposable $\V_i$ is obtained as the number of dots at the vertex $i$ which we call ``fixed vertices''. The multiplicity of the indecomposable $\U_{i,j}$ is given as the number of arrows $j\rightarrow i$. 
We define $\eolp(X)$ to be the enhanced oriented link pattern corresponding to both the isomorphism class of \mbox{$X\in\rep^{\inj}_{K}(\Q_p,I_2)(\dfp)$} and the $P$-orbit of $X\in\N_n^{(2)}$. Furthermore, we say that a matrix as in Theorem \ref{paraboliccase} (3.) is \textit{in normal form} for the $P$-action. Then the set of matrices in $P$-normal form is defined to be $R_P$.\\[1ex]
An \textit{oriented link pattern} of size $n$ is an  enhanced oriented link pattern of type $(1,\punkte ,1)$. Thus, every vertex is incident with at most one arrow.
The concrete classification of the Borel-orbits is then given by the  oriented link patterns of size $n$ and is easily obtained from Theorem \ref{paraboliccase} (see, for the detailed proof, \cite[Theorem 3.4]{BoRe}). 
As before, we define $\olp(X)$ to be the oriented link pattern corresponding to both the isomorphism class of \mbox{$X\in\rep^{inj}_{K}(\Q_n,I_2)(\dfs)$} and the $B$-orbit of $X\in\N_n^{(2)}$.
\subsubsection[Interrelation between B- and P-orbits of 2-nilpotent matrices]{Interrelation between $B$-orbits and $P$-orbits in $\N_n^{(2)}$}
Our aim is to verify an algorithm in order to determine each $B$-orbit contained in a given $P$-orbit.\\[1ex]
The idea is the following: Since each $P$-orbit is represented by a matrix $N\in K^{p\times p}$ in normal form, we can show that all $B$-orbits contained in this $P$-orbit are (as $B$-orbits) represented by matrices, which are obtained by extending $N$ to matrices in $K^{n\times n}$ and thereby translating and interpreting the entries of $N$. In this way, we obtain the above mentioned algorithm and a precise classification.\\[1ex]
Let $N=(n_{i,j})_{i,j}\in K^{n\times n}$, then define its inner sum to be \[\textrm{sum}_{i,j}(N)\coloneqq\sum_{\begin{subarray}{l}
d_{i-1}< x\leq d_i \\ 
d_{j-1}< y\leq d_j \end{subarray}} n_{x,y}.\]
 Let $\textbf{B}\coloneqq(e_1,\punkte,e_n)$ be the basis of coordinate vectors of $K^n$.
\begin{proposition}\label{pconj}
 Two matrices $N$ and $N'$ in $R_{B}$ are $P$-conjugate if and only if $ \textrm{sum}_{i,j}(N)= \textrm{sum}_{i,j}(N')$ for $i,j\in\{1,\punkte, p\}$.
\end{proposition}
\begin{proof} A matrix $S\in P$ with $S^{-1}\cdot N\cdot S=N'$ is induced by a permutation of $\textbf{B}$, say \[\sigma\textbf{B}\coloneqq(e_{\sigma(1)},\punkte, e_{\sigma(n)}),\] such that if $d_{i-1}< x\leq d_i$, then $d_{i-1} <\sigma(x) \leq d_i$ for all $i\in\{1,\punkte, p\}$.\\[1ex] 
Let $i$ and $j$ be two indices, such that $x\coloneqq \textrm{sum}_{i,j}(N)> \textrm{sum}_{i,j}(N')$ and assume there is a matrix $S\in P$ with $S^{-1}\cdot N\cdot S=N'$. Denote the corresponding non-zero entries of $N$ by $(i_s,j_s)$ for $1\leq s\leq x$; they fulfill $d_{i-1}< i_s\leq d_i$ and   $d_{j-1}< j_s\leq d_j$. Of course, $N\cdot e_{j_s}=e_{i_s}$ and due to $d_{i-1}<\sigma(i_s) \leq d_i$ and $d_{j-1}<\sigma(j_s) \leq d_j$, we obtain $x\leq \textrm{sum}_{i,j}(N')$, a contradiction.\\[1ex] 
Given $N$ and $N'$ in $R_{B}$ fulfilling $ \textrm{sum}_{i,j}(N)= \textrm{sum}_{i,j}(N')$ for $i,j\in\{1,\punkte, p\}$, we have to define a matrix $S\in P$ such that $S^{-1}\cdot N\cdot S =N'$. We, therefore, define a permutation $\sigma\in S_n$, such that the $i$-th column $S_{\cdot,i}$ of $S$ equals $e_{\sigma(i)}$. Without loss of generality we assume the oriented link patterns of $N$ and $N'$ to have $x$ arrows.\\[0.5ex]
First, we define $\sigma$ on fixed vertices.\\ 
Let $\Fa_{i}$  be the set of fixed vertices $f$ with $d_{i-1} < f\leq d_i$ in $\olp(N)$ and $\Fa'_{i}$ be the set of fixed vertices $f'$ with $d_{i-1} < f'\leq d_i$ in $\olp(N')$. Of course, the number of elements in $\Fa_i$ and $\Fa'_i$ coincides for all $i$. Given $\Fa_{i}=\{f_1,\punkte,f_{l_i}\}$ and $\Fa'_{i}=\{f'_1,\punkte,f'_{l_i}\}$, we define $\sigma(f'_k)=f_k$ for all $1\leq k\leq l_i$.\\[0.5ex]
Next, we define $\sigma$ on the source vertices of $\olp(N')$.Let $\Sa_i$ be the set of source vertices $s$ with $d_{i-1}< s\leq d_i$ in $\olp(N)$ and $\Sa'_i$ be the set of source vertices $s'$ with $d_{i-1}< s'\leq d_i$ in $\olp(N')$. We order them in the following way:\\[0.5ex]
 Let $(\Sa_{i})_{j}$ be the set of source vertices of arrows with targets $t$, such that $d_{j-1}< t\leq d_j$ in $\olp(N)$ and let $(\Sa'_{i})_{j}$ be the set of source vertices of arrows with targets $t'$, such that $d_{j-1}< t'\leq d_j$ in $\olp(N')$.
 Of course, the number of elements in $(\Sa_i)_j$ and $(\Sa'_i)_j$ coincides for $i,j\in\{1,\punkte, p\}$.\\
Given $(\Sa_{i})_j=\{s_1,\punkte,s_l\}$ and $(\Sa'_{i})_j=\{s'_1,\punkte,s'_l\}$, define $\sigma(s'_k)=s_k$ for all $k\in\{1,\punkte, l\}$.\\[0.5ex]
Finally, we define $\sigma$ on target vertices.\\
 Let $y'\in(\Sa'_i)_j$ be mapped to $y\in(\Sa_i)_j$ by $\sigma$. Let $x$ be the target of the arrow $y\rightarrow x$ in $\olp(N)$ and $x'$ be the target of the arrow $y'\rightarrow x'$ in $\olp(N')$. Then we define $\sigma(x')=x$. We have, thus, defined $\sigma$ on each vertex of the oriented link pattern and, therefore, on $S_n$. In the same way, we have defined the aforementioned basis $\sigma \textbf{B}=(e_{\sigma(i)})_{1\leq i\leq n}$.\\[1ex]
It now suffices to show $S^{-1}\cdot N\cdot S=N'$, that is, the representing matrix $M_{\sigma \textbf{B}}^{\sigma \textbf{B}}(l_{N})$ equals $N'$, here we denote by $l_N$ and $l_{N'}$ the induced linear maps.\\[0.5ex]
 If $i$ is a fixed vertex in $\olp(N')$, then $\sigma(i)$ is a fixed vertex in $\olp(N)$ and $N e_{\sigma(i)}=0$. Then the $i-$th column of $M_{\sigma \textbf{B}}^{\sigma \textbf{B}}(l_{N})$ as well as of $N'$ equals $0$.\\[0.5ex]
 If $i$ is a source vertex of an arrow in $\olp(N')$ with a target $j'$, then $\sigma(i)$ is a source vertex of an arrow in $\olp(N)$ with a target $j$. Thus, $N\cdot e_{\sigma(i)}=e_{j}$ and since $\sigma(j')=j$,  the $i-$th column of $M_{\sigma \textbf{B}}^{\sigma \textbf{B}}(l_{N})$ and of $N'$ coincide.\\[0.5ex]
If $i$ is a target vertex of an arrow in $\olp(N')$ with a source $j$, then $\sigma(i)$ is a target vertex of an arrow in $\olp(N)$ with a source $i'$. Thus, $N e_{\sigma(i)}=0$ and the $i-$th column of $N'$ equals $0$ as well.
\end{proof}
Note that the description of the $P$-orbits can also be deduced directly from the bijection given in \ref{bijection}. The proof of the theorem however gives an explicit conjugation matrix and therefore presents more details about the connection.\\[1ex]
We have proven an explicit description of the $P$-orbits and derive a natural algorithm to obtain each $B$-orbit contained in a given $P$-orbit. The interpretation in terms of oriented link patterns is quite easy.\\[1ex]
 Given an enhanced oriented link pattern of $k$ vertices, we construct oriented link patterns belonging to the $P$-orbit as follows:\\
 We draw $n$ vertices numbered by $1$, $2$ up to $n$, such that we mark the first $b_1$ vertices, then the vertices $b_1+1$ up to $b_1+b_2$ and so on. In this way, we obtain $n$ numbered vertices which are ordered in $p$ sets by the block sizes of the parabolic. Now all oriented link patterns have to be constructed, such that the number of arrows from the $j$-th tuple of vertices to the $i$-th tuple of vertices equals the number of arrows from $j$ to $i$ in the enhanced oriented link pattern. In this way, it becomes obvious why it is necessarily allowed to draw loops in an enhanced oriented link pattern.
\begin{example}
 Consider $n=4, p=2$ and the parabolic $P$ of block sizes $(3,1)$ with
\[R_{P}=\left\lbrace A:=\left( \begin{array}{ll}1 & 1 \\ 0 & 0  \end{array}\right) ,\left( \begin{array}{ll} 1 & 0 \\ 1 & 0  \end{array}\right), \left( \begin{array}{ll} 1 & 0 \\ 0 & 0  \end{array}\right), \left( \begin{array}{ll} 0 & 1 \\ 0 & 0  \end{array}\right), \left( \begin{array}{ll} 0 & 0 \\ 1 & 0\end{array}\right), \left( \begin{array}{ll} 0 & 0 \\ 0 & 0  \end{array}\right)\right\rbrace. \]
We discuss the $P$-orbit of $A$ with the enhanced oriented link pattern {\small
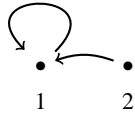
\[\begin{tikzpicture}
  \matrix (m) [matrix of math nodes, row sep=0.1em, column sep=2.5em, text height=1.5ex, text depth=0.25ex]
 { \bullet & \bullet \\
1& 2\\};
\path[->,font=\scriptsize]
(m-1-2) edge [thick,bend right=20] (m-1-1)
(m-1-1) edge [thick,loop] (m-1-1);
  \end{tikzpicture}\]}
and express a system of representatives of the Borel-orbits contained in it. These are obtained from $\eolp(A)$:
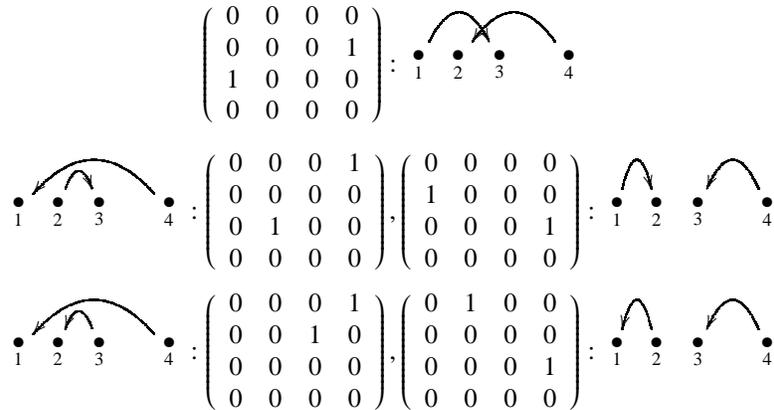
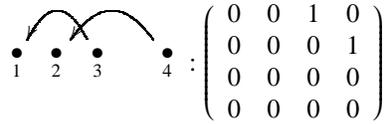
\[\left( \begin{array}{llll}
0 & 0 & 0 & 0 \\ 
0 & 0 & 0 & 1 \\ 
1 & 0 & 0 & 0 \\ 
0 & 0 & 0 & 0
        \end{array}\right): \xygraph{ !{<0cm,0cm>;<1cm,0cm>:<0cm,1cm>::} 
!{(-0.8,0) }*+{\underset{1}\bullet}="1" 
!{(-0.275,0)}*+{\underset{2}\bullet}="2"
!{(0.275,0)}*+{\underset{3} \bullet}="3" 
!{(1.2,0) }*+{\underset{4}\bullet}="4"
"1":@/^0.7cm/"3" 
"4":@/_0.7cm/"2" 
 }\]
\[ \xygraph{ !{<0cm,0cm>;<1cm,0cm>:<0cm,1cm>::} 
!{(-0.8,0) }*+{\underset{1}\bullet}="1" 
!{(-0.275,0)}*+{\underset{2}\bullet}="2"
!{(0.275,0)}*+{\underset{3} \bullet}="3" 
!{(1.2,0) }*+{\underset{4}\bullet}="4"
"2":@/^0.55cm/"3" 
"4":@/_0.7cm/"1" 
 }:\left( \begin{array}{llll}
0 & 0 & 0 & 1 \\ 
0 & 0 & 0 & 0 \\ 
0 & 1 & 0 & 0 \\ 
0 & 0 & 0 & 0
        \end{array}\right), \left( \begin{array}{llll}
0 & 0 & 0 & 0 \\ 
1 & 0 & 0 & 0 \\ 
0 & 0 & 0 & 1 \\ 
0 & 0 & 0 & 0
        \end{array}\right):\xygraph{ !{<0cm,0cm>;<1cm,0cm>:<0cm,1cm>::} 
!{(-0.8,0) }*+{\underset{1}\bullet}="1" 
!{(-0.275,0)}*+{\underset{2}\bullet}="2"
!{(0.275,0)}*+{\underset{3} \bullet}="3" 
!{(1.2,0) }*+{\underset{4}\bullet}="4"
"1":@/^0.7cm/"2" 
"4":@/_0.7cm/"3" 
 }\]
\[\xygraph{ !{<0cm,0cm>;<1cm,0cm>:<0cm,1cm>::} 
!{(-0.8,0) }*+{\underset{1}\bullet}="1" 
!{(-0.275,0)}*+{\underset{2}\bullet}="2"
!{(0.275,0)}*+{\underset{3} \bullet}="3" 
!{(1.2,0) }*+{\underset{4}\bullet}="4"
"4":@/_0.7cm/"1" 
"3":@/_0.55cm/"2" 
 }:\left( \begin{array}{llll}
0 & 0 & 0 & 1 \\ 
0 & 0 & 1 & 0 \\ 
0 & 0 & 0 & 0 \\ 
0 & 0 & 0 & 0
        \end{array}\right),  \left( \begin{array}{llll}
0 & 1 & 0 & 0 \\ 
0 & 0 & 0 & 0 \\ 
0 & 0 & 0 & 1 \\ 
0 & 0 & 0 & 0
        \end{array}\right):\xygraph{ !{<0cm,0cm>;<1cm,0cm>:<0cm,1cm>::} 
!{(-0.8,0) }*+{\underset{1}\bullet}="1" 
!{(-0.275,0)}*+{\underset{2}\bullet}="2"
!{(0.275,0)}*+{\underset{3} \bullet}="3" 
!{(1.2,0) }*+{\underset{4}\bullet}="4"
"2":@/_0.7cm/"1" 
"4":@/_0.7cm/"3" 
 }\]
\[\xygraph{ !{<0cm,0cm>;<1cm,0cm>:<0cm,1cm>::} 
!{(-0.8,0) }*+{\underset{1}\bullet}="1" 
!{(-0.275,0)}*+{\underset{2}\bullet}="2"
!{(0.275,0)}*+{\underset{3} \bullet}="3" 
!{(1.2,0) }*+{\underset{4}\bullet}="4"
"3":@/_0.7cm/"1" 
"4":@/_0.7cm/"2" 
 }: \left( \begin{array}{llll}
0 & 0 & 1 & 0 \\ 
0 & 0 & 0 & 1 \\ 
0 & 0 & 0 & 0 \\ 
0 & 0 & 0 & 0
        \end{array}\right)\]
\end{example}
We denote the parabolic subgroup of block sizes $(i_1,\punkte,i_k)$ by $P_{i_1,\punkte,i_k}$.
\begin{example} Let $n=3$ and consider the actions of the Borel subgroup $B$, the parabolic subgroups $P_{2,1}$ (and $P_{1,2}$, which is the symmetric case) and the general linear group $\GL_3$ on the variety $\N_3^{(2)}$ of $2$-nilpotent matrices.\\[1ex]
Oriented link patterns representing the $B_3$-orbits in $\N_3^{(2)}$:
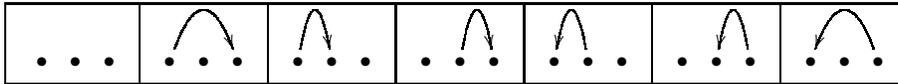
\begin{center}
{\small\begin{tabular}{|c|c|c|c|c|c|c|}
\hline 
\parbox[0pt][4.9em][c]{0cm}{}
 $\xygraph{ !{<0cm,0cm>;<0.75cm,0cm>:<0cm,0.75cm>::} 
!{(-0.6,0) }*+{\bullet}="1" 
!{(0,0) }*+{\bullet}="2" 
!{(0.6,0) }*+{\bullet}="3" 
 }$  &  $\xygraph{ !{<0cm,0cm>;<0.75cm,0cm>:<0cm,0.75cm>::} 
!{(-0.6,0) }*+{\bullet}="1" 
!{(0,0) }*+{\bullet}="2" 
!{(0.6,0) }*+{\bullet}="3" 
"1":@/^0.7cm/"3" 
 }$ & $\xygraph{ !{<0cm,0cm>;<0.75cm,0cm>:<0cm,0.75cm>::} 
!{(-0.6,0)}*+{\bullet}="1" 
!{(0,0)}*+{\bullet}="2" 
!{(0.6,0)}*+{\bullet}="3" 
"1":@/^0.7cm/"2" 
 }$ &  $ \xygraph{ !{<0cm,0cm>;<0.75cm,0cm>:<0cm,0.75cm>::} 
!{(-0.6,0) }*+{\bullet}="1" 
!{(0,0) }*+{\bullet}="2" 
!{(0.6,0) }*+{\bullet}="3" 
"2":@/^0.7cm/"3" 
 }$ 
 & $\xygraph{ !{<0cm,0cm>;<0.75cm,0cm>:<0cm,0.75cm>::} 
!{(-0.6,0) }*+{\bullet}="1" 
!{(0,0) }*+{\bullet}="2" 
!{(0.6,0) }*+{\bullet}="3" 
"2":@/_0.7cm/"1" 
 }$ &  $ \xygraph{ !{<0cm,0cm>;<0.75cm,0cm>:<0cm,0.75cm>::} 
!{(-0.6,0) }*+{\bullet}="1" 
!{(0,0) }*+{\bullet}="2" 
!{(0.6,0) }*+{\bullet}="3" 
"3":@/_0.7cm/"2" 
 }$ &
$\xygraph{ !{<0cm,0cm>;<0.75cm,0cm>:<0cm,0.75cm>::} 
!{(-0.6,0) }*+{\bullet}="1" 
!{(0,0) }*+{\bullet}="2" 
!{(0.6,0) }*+{\bullet}="3" 
"3":@/_0.7cm/"1"}$ \\[-3.6ex] \hline
\end{tabular}}
\end{center}
Enhanced oriented link patterns representing the $P_{2,1}$-orbits in $\N_3^{(2)}$ and the corresponding oriented link patterns:
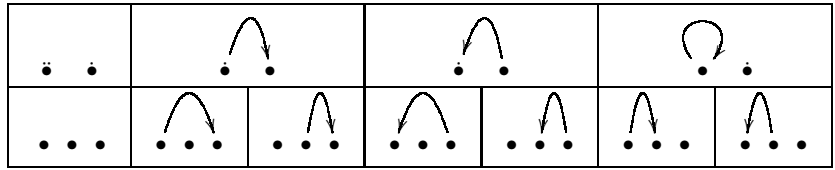
\begin{center}
{\small\begin{tabular}{|c|c|c|c|c|c|c|}
 \hline
$ \xygraph{ !{<0cm,0cm>;<0.75cm,0cm>:<0cm,0.75cm>::} 
!{(-0.4,0.04) }*+{\ddot{\bullet}}="1" 
!{(0.4,0.04) }*+{\dot{\bullet}}="3" 
 }$ &  \multicolumn{2}{c|}{$\xygraph{ !{<0cm,0cm>;<0.75cm,0cm>:<0cm,0.75cm>::} 
!{(-0.4,0.04) }*+{\dot{\bullet}}="1" 
!{(0.4,0) }*+{\bullet}="3" 
"1":@/^0.7cm/"3" }$}    &   
 \multicolumn{2}{c|}{$ \xygraph{ !{<0cm,0cm>;<0.75cm,0cm>:<0cm,0.75cm>::} 
!{(-0.4,0.04) }*+{\dot{\bullet}}="1" 
!{(0.4,0) }*+{\bullet}="3" 
"3":@/_0.7cm/"1" }$}   &    \multicolumn{2}{c|}{$ \xygraph{ !{<0cm,0cm>;<0.75cm,0cm>:<0cm,0.75cm>::} 
!{(-0.4,0) }*+{\bullet}="1" 
!{(+0.4,0.04) }*+{\dot{\bullet}}="2" 
"1":@(ul,ur) "1"
}$} 
 \\ \hline
\parbox[0pt][4.9em][c]{0cm}{}
 $\xygraph{ !{<0cm,0cm>;<0.75cm,0cm>:<0cm,0.75cm>::} 
!{(-0.5,0) }*+{\bullet}="1" 
!{(0,0) }*+{\bullet}="2" 
!{(0.5,0) }*+{\bullet}="3" 
 }$  &   $\xygraph{ !{<0cm,0cm>;<0.75cm,0cm>:<0cm,0.75cm>::} 
!{(-0.5,0) }*+{\bullet}="1" 
!{(0,0) }*+{\bullet}="2" 
!{(0.5,0) }*+{\bullet}="3" 
"1":@/^0.7cm/"3" 
 }$ &  $ \xygraph{ !{<0cm,0cm>;<0.75cm,0cm>:<0cm,0.75cm>::} 
!{(-0.5,0) }*+{\bullet}="1" 
!{(0,0) }*+{\bullet}="2" 
!{(0.5,0) }*+{\bullet}="3" 
"2":@/^0.7cm/"3" 
 }$  & $\xygraph{ !{<0cm,0cm>;<0.75cm,0cm>:<0cm,0.75cm>::} 
!{(-0.5,0) }*+{\bullet}="1" 
!{(0,0) }*+{\bullet}="2" 
!{(0.5,0) }*+{\bullet}="3" 
"3":@/_0.7cm/"1" 
 }$
 &  $\xygraph{ !{<0cm,0cm>;<0.75cm,0cm>:<0cm,0.75cm>::} 
!{(-0.5,0) }*+{\bullet}="1" 
!{(0,0) }*+{\bullet}="2" 
!{(0.5,0) }*+{\bullet}="3" 
"3":@/_0.7cm/"2" 
 }$ &  $\xygraph{ !{<0cm,0cm>;<0.75cm,0cm>:<0cm,0.75cm>::} 
!{(-0.5,0)}*+{\bullet}="1" 
!{(0,0)}*+{\bullet}="2" 
!{(0.5,0)}*+{\bullet}="3" 
"1":@/^0.7cm/"2" 
 }$ & 
 $\xygraph{ !{<0cm,0cm>;<0.75cm,0cm>:<0cm,0.75cm>::} 
!{(-0.5,0) }*+{\bullet}="1" 
!{(0,0) }*+{\bullet}="2" 
!{(0.5,0) }*+{\bullet}="3" 
"2":@/_0.7cm/"1" 
 }$ \\[-3.6ex]\hline
\end{tabular}}
\end{center}
\vspace{0.5ex}
Enhanced oriented link patterns representing the $\GL_3$-orbits in $\N_3^{(2)}$ and the corresponding oriented link patterns:
\begin{center}
{\small\begin{tabular}{|c|c|c|c|c|c|c|}\hline
$\xygraph{ !{<0cm,0cm>;<0.75cm,0cm>:<0cm,0.75cm>::} 
!{(0,0.04) }*+{\dddot{\bullet}}="1" 
 }$ &  \multicolumn{6}{c|}{$\xygraph{ !{<0cm,0cm>;<0.75cm,0cm>:<0cm,0.75cm>::} 
!{(0,0) }*+{\dot{\bullet}}="1" 
"1":@(ul,ur) "1"
 }$} 
 \\ \hline\parbox[0pt][4.9em][c]{0cm}{}
$\xygraph{ !{<0cm,0cm>;<0.75cm,0cm>:<0cm,0.75cm>::} 
!{(-0.5,0) }*+{\bullet}="1" 
!{(0,0) }*+{\bullet}="2" 
!{(0.5,0) }*+{\bullet}="3" 
 }$  &   $\xygraph{ !{<0cm,0cm>;<0.75cm,0cm>:<0cm,0.75cm>::} 
!{(-0.5,0) }*+{\bullet}="1" 
!{(0,0) }*+{\bullet}="2" 
!{(0.5,0) }*+{\bullet}="3" 
"1":@/^0.7cm/"3" 
 }$ & $\xygraph{ !{<0cm,0cm>;<0.75cm,0cm>:<0cm,0.75cm>::} 
!{(-0.5,0)}*+{\bullet}="1" 
!{(0,0)}*+{\bullet}="2" 
!{(0.5,0)}*+{\bullet}="3" 
"1":@/^0.7cm/"2" 
 }$   & $\xygraph{ !{<0cm,0cm>;<0.75cm,0cm>:<0cm,0.75cm>::} 
!{(-0.5,0) }*+{\bullet}="1" 
!{(0,0) }*+{\bullet}="2" 
!{(0.5,0) }*+{\bullet}="3" 
"3":@/_0.7cm/"1" 
 }$
 &   $\xygraph{ !{<0cm,0cm>;<0.75cm,0cm>:<0cm,0.75cm>::} 
!{(-0.5,0) }*+{\bullet}="1" 
!{(0,0) }*+{\bullet}="2" 
!{(0.5,0) }*+{\bullet}="3" 
"2":@/_0.7cm/"1" 
 }$&  $ \xygraph{ !{<0cm,0cm>;<0.75cm,0cm>:<0cm,0.75cm>::} 
!{(-0.5,0) }*+{\bullet}="1" 
!{(0,0) }*+{\bullet}="2" 
!{(0.5,0) }*+{\bullet}="3" 
"2":@/^0.7cm/"3" 
 }$ & 
$\xygraph{ !{<0cm,0cm>;<0.75cm,0cm>:<0cm,0.75cm>::} 
!{(-0.5,0) }*+{\bullet}="1" 
!{(0,0) }*+{\bullet}="2" 
!{(0.5,0) }*+{\bullet}="3" 
"3":@/_0.7cm/"2" 
 }$ \\[-3.8ex] \hline

\end{tabular}}\end{center} 
\end{example}
\subsubsection[U-orbits of 2-nilpotent matrices and labelled oriented link patterns]{$U$-orbits in $\N_n^{(2)}$ and labelled oriented link patterns}
For completeness, we discuss the orbits of the unipotent subgroup $U$ in $\N_n^{(2)}$ briefly. The action is of infinite type, but the orbits can be rederived from the classification of the $B$-orbits in Theorem \ref{paraboliccase}.\\[1ex]
 A labelled oriented link pattern of size $n$ is defined to be a tuple $\olp_{\lambda}\coloneqq(\olp,\lambda)$ where $\olp$ is an oriented link pattern of size $n$ and $\lambda\in (K^*)^s$, such that the arrow $j_k\rightarrow i_k$ is labelled by $\lambda_k$; here $s$ equals the number of arrows in $\olp$.
We can illustrate the labelled oriented link pattern given by $\lambda=(3,6,1)$ and the oriented link pattern $(3,1)(5,6)(2,7)$:
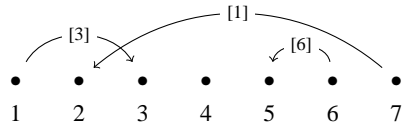
\begin{center}{\small\begin{tikzpicture}[descr/.style={fill=white,inner sep=2.5pt}]
  \matrix (m) [matrix of math nodes, row sep=0.01em, column sep=1.5em, text height=1.5ex, text depth=0.25ex]
 {\bullet & \bullet & \bullet & \bullet & \bullet & \bullet & \bullet\\
1 & 2 & 3 & 4 & 5 & 6 & 7\\};
\path[->,font=\scriptsize]
(m-1-1) edge [bend left=60] node[descr] {$[3]$} (m-1-3)
(m-1-6) edge [bend right=80] node[descr] {$[6]$}  (m-1-5)
(m-1-7) edge [bend right=40] node[descr] {$[1]$}  (m-1-2);
  \end{tikzpicture}.}\end{center} 

Given a labelled oriented link pattern $\olp_{\lambda}$  of size $n$, we can define the matrix $N(\olp_{\lambda})\in \N_n^{(2)}$ by 
\[N(\olp_{\lambda})_{i,j}= \left\{ \begin{array}{ll} \lambda_k, & \hbox{$ \textrm{if}~ i=i_k~ \textrm{and}~j=j_k$;} \\ 
0, & \hbox{$ \textrm{otherwise}$.} \end{array} \right.\]
Denote furthermore $N(\olp)\coloneqq N(\olp_{(1,\punkte,1)})$.
\begin{lemma}\label{classU}
 There are natural bijections between
\begin{enumerate}
\item  $U$-orbits in $\N_n^{(2)}$,
 \item matrices $N(\olp_{\lambda})$ where $\olp_{\lambda}$ is a labelled oriented link pattern of size $n$ and
\item labelled oriented link patterns of size $n$.
\end{enumerate}
\end{lemma} 
\begin{proof}
 The bijection between 2. and 3. is immediately clear. The bijection between 1. and 2. is a direct consequence of Theorem \ref{paraboliccase}.
\end{proof}
Each $U$-orbit is closed itself, see for example \cite{Kr}.

\section[P-orbit closures in the variety of 2-nilpotent matrices]{$P$-orbit closures in $\N_n^{(2)}$}\label{closures}
Given representations $M,M'\in \rep_{K}(\Q_p,I_2)$, we set $[M, M'] \coloneqq \dim_{K} \Hom(M, M' )$.
Proposition \ref{dimhom} can be found in \cite[Lemma 4.2]{BoRe}; here $\delta_{x\leq y} \coloneqq 1$ if $x\leq y$ and $\delta_{x\leq y} \coloneqq 0$ otherwise.
\begin{proposition}\label{dimhom}
 Let $i, j, k, l \in \{1,\punkte, p\}$. Then
\begin{enumerate}
\item $[\V_k , \V_i ] = [\V_k , \U_{i,j} ] =\delta_{i\leq k}$,
\item $[\U_{k,l} , \V_i ] = \delta_{i\leq l}$,
\item $[\U_{k,l} , \U_{i,j} ] = \delta_{i\leq l} + \delta_{j\leq l} \cdot \delta_{i\leq k}$,
\end{enumerate}
\end{proposition}
These dimensions are linked with (enhanced) oriented link patterns as follows (see \cite{BoRe}).
\begin{proposition}\label{combihom}
 Let $M\in \rep^{\inj}_{K}(\Q_p,I_2)(\dfp)$ and let $i, j, k, l \in \{1,\punkte, p\}$. Then by considering $X:=\eolp(M)$:\\[1ex]
 1. $a_k(M)\coloneqq[\V_k,M]= \sharp\{ fixed~vertices~\leq k~{\rm in}~X\}~+~\sharp\{targets~of~arrows~\leq k~{\rm in}~X\}$, \\[1ex]
2. $b_{k,l}(M)\coloneqq[\U_{k,l},M]= a_l(M) + \sharp\{ \textrm{arrows~with~source}~\leq l~  \textrm{and~target}~\leq k~{\rm in}~X\}$,\\[1ex]
3. $\overline{a_i}(M)\coloneqq[M,\V_i]= \sharp\{ \textrm{fixed~vertices}~\geq i~{\rm in}~X\}~+~\sharp\{ \textrm{sources~of~arrows}~\geq i~{\rm in}~X\}$,\\[1ex]
4. $\overline{b_{i,j}}(M)\coloneqq[M,\U_{i,j}]= \overline{a_i}(M) + \sharp\{ \textrm{arrows~with~source}~\geq j~  \textrm{and~target}~\geq i~{\rm in}~X\}$.
\end{proposition}
For two representations $M=\bigoplus_{i,j=1}^p\U_{i,j}^{m_{i,j}}\oplus \bigoplus_{i=1}^p\V_i^{n_i}$ and $M'=\bigoplus_{i,j=1}^p\U_{i,j}^{m'_{i,j}}\oplus \bigoplus_{i=1}^p\V_i^{n'_i}$ in $\rep^{\inj}_{K}(\Q_p,I_2)(\dfp)$, we obtain
\[[M,M']=\sum_{i,j=1}^p m_{i,j}b_{i,j}(M')+ \sum_{k=1}^p n_k a_k(M')=\sum_{i,j=1}^p m'_{i,j}\overline{b_{i,j}}(M)+ \sum_{k=1}^p n'_k \overline{a_k}(M).\]
Let $N\in\N_n^{(2)}$ be a $2$-nilpotent matrix that corresponds to the representation $M$ via the bijection of Lemma \ref{bijection}.
\begin{proposition}\label{stabpar}
\[\dim \overline{P.N}=\dim P.N=\sum\limits_{i=1}^p \sum\limits_{x=1}^i (b_i\cdot b_x)-\sum\limits_{i,j=1}^p m_{i,j}b_{i,j}(N)-\sum\limits_{i=1}^p n_i a_i(N). \]
\end{proposition}
\begin{proof}The equalities
\[ \dim P.N~= \dim P-\dim\Iso_{P}(N) = \dim P- \dim\Iso_{\GL_{\dfp}}(M) = \dim P - [M,M] \]
yield the claim, here $\Iso_{\GL_{\dfp}}(M)$ is the isotropy group of $m\in R_{\dfp}^{\inj}(\Q_p,I_2)$ in $\GL_{\dfp}$.
\end{proof}
Let $M$ and $M'$ be two representations in $\rep_{K}(\Q_p,I_2)(\df)$.  Since the correspondence of Lemma \ref{bijection} preserves orbit closure relations, we know that $M\leq_{\rm deg}M'$ if and only if the corresponding $2$-nilpotent matrices, denoted by $N=(m_{i,j})_{i,j}$ and $N'=(m'_{i,j})_{i,j}$, respectively, fulfill $P.N'\subset\overline{P.N}$ in $\N_n^{(2)}$. The following theorem is a slightly generalized version of \cite[Theorem 4.3]{BoRe}; the proof is similar, though.
\begin{theorem}\label{pq} We have $M\leq_{\rm deg} M'$ if and only if $a_k(M)\leq a_k(M')$ and $b_{k,l}(M)\leq b_{k,l}(M')$ for all $k,l\in\{1,\punkte,p\}$.
\end{theorem}

We describe all minimal, disjoint degenerations analogously to \cite[Theorem 4.6]{BoRe}, where they were described for the Borel-action.
\begin{theorem}\label{mindisjpar}
Let $D<_{\mdeg}D'$ be a minimal, disjoint degeneration in $\rep^{\inj}_{K}(\Q_p,I_2)$. Then it either appears in \cite[Theorem 4.6]{BoRe} or in one of the following chains.

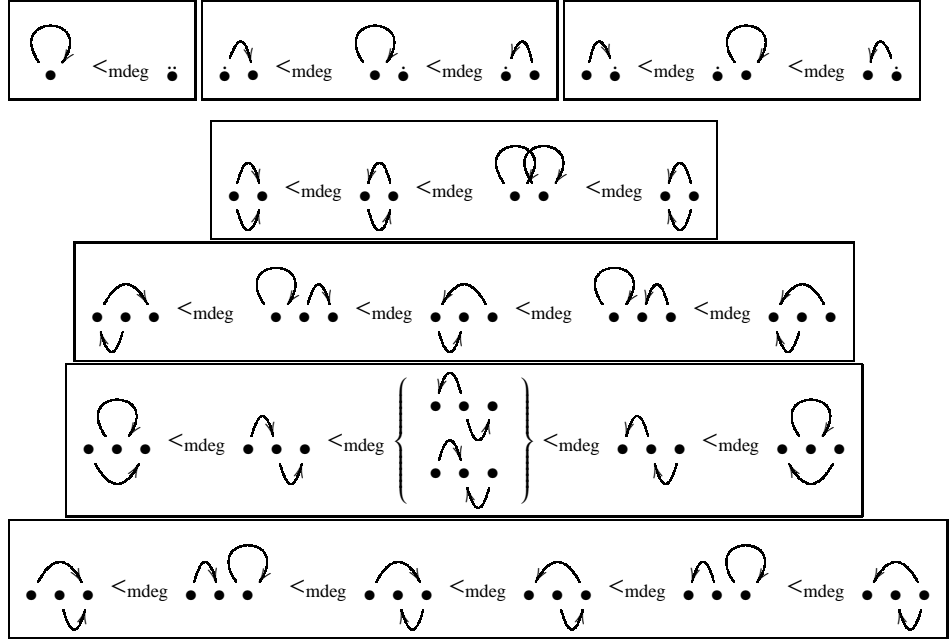
\fbox{$\xygraph{ !{<0cm,-0.13cm>;<0.63cm,-0.13cm>:<0cm,0.63cm>::} 
!{(0,0)}*+{\bullet}="1" 
"1":@(ul,ur) "1" }<_{\mdeg}
\xygraph{ !{<0cm,-0.13cm>;<0.63cm,-0.13cm>:<0cm,0.63cm>::} 
!{(0,0.03) }*+{\ddot{\bullet}}="1" 
 }$} \fbox{$\xygraph{ !{<0cm,-0.13cm>;<0.63cm,-0.13cm>:<0cm,0.63cm>::} 
!{(-0.6,0.03) }*+{\dot{\bullet}}="1" 
!{(0,0) }*+{\bullet}="2" 
"1":@/^0.45cm/"2" 
 }<_{\mdeg}\xygraph{ !{<0cm,-0.13cm>;<0.63cm,-0.13cm>:<0cm,0.63cm>::} 
!{(-0.6,0)}*+{\bullet}="1" 
!{(0,0.03)}*+{\dot{\bullet}}="2" 
"1":@(ul,ur) "1" }<_{\mdeg}\xygraph{ !{<0cm,-0.13cm>;<0.63cm,-0.13cm>:<0cm,0.63cm>::} 
!{(-0.6,0.03) }*+{\dot{\bullet}}="1" 
!{(0,0) }*+{\bullet}="2" 
"2":@/_0.45cm/"1" 
 }$} 
\fbox{$\xygraph{ !{<0cm,-0.13cm>;<0.63cm,-0.13cm>:<0cm,0.63cm>::} 
!{(-0.6,0) }*+{\bullet}="1" 
!{(0,0.03) }*+{\dot{\bullet}}="2" 
"1":@/^0.45cm/"2" 
 }<_{\mdeg}\xygraph{ !{<0cm,-0.13cm>;<0.63cm,-0.13cm>:<0cm,0.63cm>::} 
!{(-0.6,0.03)}*+{\dot{\bullet}}="1"
!{(0,0) }*+{\bullet}="2"
 "2":@(ul,ur) "2" }<_{\mdeg}\xygraph{ !{<0cm,-0.13cm>;<0.63cm,-0.13cm>:<0cm,0.63cm>::} 
!{(-0.6,0) }*+{\bullet}="1" 
!{(0,0.03) }*+{\dot{\bullet}}="2" 
"2":@/_0.45cm/"1" 
 }$}\begin{center}
\fbox{$\xygraph{ !{<0cm,-0.13cm>;<0.63cm,-0.13cm>:<0cm,0.63cm>::} 
!{(-0.6,0) }*+{\bullet}="1" 
!{(0,0) }*+{\bullet}="2" 
"1":@/^0.45cm/"2" 
"1":@/_0.45cm/"2" 
 }<_{\mdeg} \xygraph{ !{<0cm,-0.13cm>;<0.63cm,-0.13cm>:<0cm,0.63cm>::} 
!{(-0.6,0) }*+{\bullet}="1" 
!{(0,0) }*+{\bullet}="2" 
"1":@/_0.45cm/"2" 
"2":@/_0.45cm/"1"}<_{\mdeg}
\xygraph{ !{<0cm,-0.13cm>;<0.63cm,-0.13cm>:<0cm,0.63cm>::} 
!{(-0.6,0)}*+{\bullet}="1"
!{(0,0) }*+{\bullet}="2"
 "1":@(ul,ur) "1"
 "2":@(ul,ur) "2" }<_{\mdeg}
\xygraph{ !{<0cm,-0.13cm>;<0.63cm,-0.13cm>:<0cm,0.63cm>::} 
!{(-0.6,0) }*+{\bullet}="1" 
!{(0,0) }*+{\bullet}="2" 
"2":@/^0.45cm/"1" 
"2":@/_0.45cm/"1" 
 }$}\\
\fbox{$\xygraph{!{<0cm,-0.13cm>;<0.63cm,-0.13cm>:<0cm,0.63cm>::} 
!{(-0.6,0) }*+{\bullet}="1" 
!{(0,0) }*+{\bullet}="2" 
!{(0.6,0) }*+{\bullet}="3" 
"1":@/^0.45cm/"3" 
"2":@/^0.45cm/"1"}<_{\mdeg} \xygraph{!{<0cm,-0.13cm>;<0.63cm,-0.13cm>:<0cm,0.63cm>::} 
!{(-0.6,0)}*+{\bullet}="1" 
!{(0,0)}*+{\bullet}="2" 
!{(0.6,0)}*+{\bullet}="3" 
 "1":@(ul,ur) "1"
"2":@/^0.45cm/"3"}
 <_{\mdeg}\xygraph{!{<0cm,-0.13cm>;<0.63cm,-0.13cm>:<0cm,0.63cm>::} 
!{(-0.6,0) }*+{\bullet}="1" 
!{(0,0) }*+{\bullet}="2" 
!{(0.6,0) }*+{\bullet}="3" 
"3":@/_0.45cm/"1" 
"1":@/_0.45cm/"2"}<_{\mdeg} \xygraph{!{<0cm,-0.13cm>;<0.63cm,-0.13cm>:<0cm,0.63cm>::} 
!{(-0.6,0) }*+{\bullet}="1"
!{(0,0) }*+{\bullet}="2"  
!{(0.6,0)}*+{\bullet}="3" 
 "1":@(ul,ur) "1"
"3":@/_0.45cm/"2"} 
 <_{\mdeg}\xygraph{!{<0cm,-0.13cm>;<0.63cm,-0.13cm>:<0cm,0.63cm>::} 
!{(-0.6,0) }*+{\bullet}="1" 
!{(0,0) }*+{\bullet}="2" 
!{(0.6,0) }*+{\bullet}="3" 
"2":@/^0.45cm/"1" 
"3":@/_0.45cm/"1"}$}\\
\fbox{$\xygraph{!{<0cm,-0.13cm>;<0.63cm,-0.13cm>:<0cm,0.63cm>::} 
!{(-0.6,0) }*+{\bullet}="1" 
!{(0,0) }*+{\bullet}="2" 
!{(0.6,0) }*+{\bullet}="3" 
 "2":@(ul,ur) "2"
"1":@/_0.45cm/"3" }<_{\mdeg}\xygraph{!{<0cm,-0.13cm>;<0.63cm,-0.13cm>:<0cm,0.63cm>::} 
!{(-0.6,0) }*+{\bullet}="1" 
!{(0,0) }*+{\bullet}="2" 
!{(0.6,0) }*+{\bullet}="3" 
"1":@/^0.45cm/"2" 
"2":@/_0.45cm/"3"}<_{\mdeg}% use packages: array
\left\lbrace \begin{array}{l}
\xygraph{!{<0cm,-0.13cm>;<0.63cm,-0.13cm>:<0cm,0.63cm>::} 
!{(-0.6,0) }*+{\bullet}="1" 
!{(0,0) }*+{\bullet}="2" 
!{(0.6,0) }*+{\bullet}="3" 
"2":@/_0.45cm/"1" 
"2":@/_0.45cm/"3"} \\
\xygraph{!{<0cm,-0.13cm>;<0.63cm,-0.13cm>:<0cm,0.63cm>::} 
!{(-0.6,0) }*+{\bullet}="1" 
!{(0,0) }*+{\bullet}="2" 
!{(0.6,0) }*+{\bullet}="3" 
"1":@/^0.45cm/"2" 
"3":@/^0.45cm/"2"}             \end{array}\right\rbrace <_{\mdeg}\xygraph{!{<0cm,-0.13cm>;<0.63cm,-0.13cm>:<0cm,0.63cm>::} 
!{(-0.6,0) }*+{\bullet}="1" 
!{(0,0) }*+{\bullet}="2" 
!{(0.6,0) }*+{\bullet}="3" 
"2":@/_0.45cm/"1" 
"3":@/^0.45cm/"2"} <_{\mdeg}\xygraph{!{<0cm,-0.13cm>;<0.63cm,-0.13cm>:<0cm,0.63cm>::} 
!{(-0.6,0) }*+{\bullet}="1" 
!{(0,0) }*+{\bullet}="2" 
!{(0.6,0) }*+{\bullet}="3" 
 "2":@(ul,ur) "2"
"3":@/^0.45cm/"1" }$}\\
\fbox{$\xygraph{!{<0cm,-0.13cm>;<0.63cm,-0.13cm>:<0cm,0.63cm>::} 
!{(-0.6,0) }*+{\bullet}="1" 
!{(0,0) }*+{\bullet}="2" 
!{(0.6,0) }*+{\bullet}="3" 
"1":@/^0.45cm/"3" 
"2":@/_0.45cm/"3"}<_{\mdeg} 
\xygraph{!{<0cm,-0.13cm>;<0.63cm,-0.13cm>:<0cm,0.63cm>::} 
!{(-0.6,0)}*+{\bullet}="1"
!{(0,0)}*+{\bullet}="2" 
!{(0.6,0)}*+{\bullet}="3" 
"1":@/^0.45cm/"2"
 "3":@(ul,ur) "3"}
 <_{\mdeg}
\xygraph{!{<0cm,-0.13cm>;<0.63cm,-0.13cm>:<0cm,0.63cm>::} 
!{(-0.6,0) }*+{\bullet}="1" 
!{(0,0) }*+{\bullet}="2" 
!{(0.6,0) }*+{\bullet}="3" 
"1":@/^0.45cm/"3" 
"3":@/^0.45cm/"2"}<_{\mdeg} 
\xygraph{!{<0cm,-0.13cm>;<0.63cm,-0.13cm>:<0cm,0.63cm>::} 
!{(-0.6,0) }*+{\bullet}="1" 
!{(0,0) }*+{\bullet}="2" 
!{(0.6,0) }*+{\bullet}="3" 
"2":@/_0.45cm/"3" 
"3":@/_0.45cm/"1"}<_{\mdeg} 
\xygraph{!{<0cm,-0.13cm>;<0.63cm,-0.13cm>:<0cm,0.63cm>::} 
!{(-0.6,0) }*+{\bullet}="1" 
!{(0,0) }*+{\bullet}="2" 
!{(0.6,0) }*+{\bullet}="3" 
"2":@/_0.45cm/"1"
 "3":@(ul,ur) "3"}
 <_{\mdeg}
\xygraph{!{<0cm,-0.13cm>;<0.63cm,-0.13cm>:<0cm,0.63cm>::} 
!{(-0.6,0) }*+{\bullet}="1" 
!{(0,0) }*+{\bullet}="2" 
!{(0.6,0) }*+{\bullet}="3" 
"3":@/^0.45cm/"2" 
"3":@/_0.45cm/"1"}$}\end{center}
\end{theorem}
These minimal, disjoint degenerations yields concrete descriptions of the orbit closures in terms of enhanced oriented link patterns right away. 
\subsubsection{Dimensions of orbits}\label{dimsubpar}
The same reasoning as in the previous section yields the following results about the dimensions of the $P$-orbits.
Let $N\in\N^{(2)}$ be a $2$-nilpotent matrix that corresponds to a representation in $\rep_{K}^{\inj}(\Q_p,I_2)(\dfp)$ via the bijection of Lemma \ref{bijection}:
\[M=\bigoplus\limits_{i,j=1}^p\U_{i,j}^{m_{i,j}}\oplus \bigoplus\limits_{i=1}^p\V_i^{n_i}.\]  
Since  $\dim\overline{\Orb_M}=\dim \Orb_M=\dim P-\dim \Iso_{P}(M)=\dim P-[M,M]$, we have
 \[\dim\overline{\Orb_M}=\dim\Orb_M=\sum\limits_{i=1}^p \left(\sum\limits_{x=1}^i b_i\right)^2-\sum\limits_{i,j=1}^p m_{i,j}b_{i,j}(M)-\sum\limits_{i=1}^p n_i a_i(M).\]
There is a unique $\GL_{\dfp}$-orbit of minimal dimension in $R_{\dfp}(\Q_p,I)$, represented by $M_0\coloneqq\bigoplus\limits_{i=1}^p\V_i^{b_i}$. 
It corresponds naturally to the $P$-orbit of minimal dimension in $\N_n^{(2)}$, which is represented by the zero-matrix and has dimension $0$. Thus,
\[\dim P.N=\sum\limits_{i=1}^p \sum\limits_{x=1}^i (b_i\cdot b_x)-\sum\limits_{i,j=1}^p m_{i,j}b_{i,j}(N)-\sum\limits_{i=1}^p n_i a_i(N). \]
We describe the open orbits for the parabolic actions. In case $\GL_n$ acts, the open orbit is clearly given by the highest rank matrices. In case of a parabolic action, the description is slightly more difficult, though.\\[1ex]
 Let $M$ be a representation in $R_{\dfp}^{\inj}(\Q_p,I_2)$ and consider the enhanced oriented link pattern corresponding to $M$. As has been seen in Proposition \ref{pconj}, this enhanced oriented link pattern can be extended to an oriented link pattern by splitting each vertex $k$ into $b_k$ vertices $k^{(1)},\punkte,k^{(b_k)}$ and drawing arrows accordingly. Without loss of generality, we denote the vertices by $1_{P},\punkte,n_{P}$ and can read off the open orbit directly.\\[1ex]
We define $\U^{P}_{i_{P},j_{P}}\coloneqq\U_{x,y}$ if there exist $1\leq s\leq b_x$ and $1\leq t\leq b_y$, such that $i_{P}=b_x^{(s)}$ and $j_{P}=b_y^{(t)}$. Furthermore, set $\V^{P}_{i_{P}}\coloneqq\V_{x}$ if there exists an integer $1\leq s\leq b_x$, such that $i_{P}=b_x^{(s)}$.
\begin{proposition}
The open orbit is represented by (1) for even integers $n$ and by (2) for odd integers $n$: 
\[(1)~M_{open}=\bigoplus\limits_{k=1}^{n/2}\U^P_{(n-k+1)_{P},k_{P}};~~ (2)~M_{open}=\bigoplus\limits_{k=1}^{(n-1)/2}\U^P_{(n-k+1)_{P},k_{P}}\oplus \V^P_{(\frac{n-1}{2}+1)_{P}}.\]
\end{proposition}
\begin{proof}  
Regardless of $n$ being even or odd, every oriented link pattern corresponding to an arbitrary representation $M\in\rep^{\inj}_{K}(\Q_p,I_2)(\dfp)$ is obtained by applying ``decreasing minimal changes'' to the oriented link pattern of $M_{open}$. Thus, for each representation $M_{open}\ncong M\in\rep^{\inj}_{K}(\Q_p,I_2)(\dfp)$, the degeneration $M_{open}<_{\deg}M$ is a proper chain of minimal degenerations and $\dim\Orb_{M_{open}}-\dim\Orb_{M}\geq 1$.
\end{proof} 
\subsubsection[Minimal degenerations in B-orbit closures]{Minimal degenerations in $B$-orbit closures}
 The key to calculating all minimal degenerations is obtained by the following proposition (see \cite[Corollary 4.5]{BoRe}).
\begin{proposition}\label{mindisj}
Let $D<_{\mdeg}D'$ be a minimal, disjoint degeneration in $\rep^{\inj}_{K}(\Q_n,I_2)$. Then either $D'$ is indecomposable or $D'\cong U\oplus V$, where $U$ and $V$ are indecomposables and there exists an exact sequence $0\rightarrow U\rightarrow D\rightarrow V\rightarrow 0$ or $0\rightarrow V\rightarrow D\rightarrow U\rightarrow 0$.
\end{proposition}
A method to construct all orbits contained in a given orbit closure is described in \cite[Theorem 4.6]{BoRe}, since Proposition \ref{mindisj} ``localizes`` the problem to sequences of changes at at most four vertices of the corresponding oriented link pattern. All these minimal, disjoint degenerations are explicitly listed (in terms of oriented link patterns as well) in \cite[Theorem 4.6]{BoRe}.\\[1ex]
Our aim is to describe all minimal degenerations in detail.
Consider an arbitrary minimal, disjoint degeneration $D<_{\mdeg}D'$ in $\rep^{\inj}_{K}(\Q_n,I_2)$. To classify the minimal degenerations in $\rep^{\inj}_{K}(\Q_n,I_2)(\dfs)$, let us (if possible) consider a representation $W$, such that $D\oplus W<_{\deg}D'\oplus W$ is a degeneration in $\rep^{\inj}_{K}(\Q_n,I_2)(\dfs)$. We give an explicit criterion as to whether this degeneration is minimal. 
\begin{theorem}\label{indecgeneral}
The degeneration $D\oplus W<_{\deg}D'\oplus W$ is minimal if and only if every indecomposable direct summand $X$ of $W$ fulfills $[X,D]-[X,D']=0$ and $[D,X]-[D',X]=0$.
\end{theorem}
\begin{proof}
Assume, the degeneration $D<_{\deg}D'$ is obtained from extensions as in Proposition \ref{mindisj}.
We extract the argumentation from \cite[Theorem 4]{Bo1}.\\[1ex]
Let $D\oplus W<_{\mdeg}U\oplus V\oplus W$ be a minimal degeneration in $\rep^{\inj}_{K}(\Q_n,I_2)(\dfs)$, such that $U$ and $V$ are indecomposables and there exists an exact sequence $0\rightarrow U\rightarrow D\rightarrow V\rightarrow 0$.\\[1ex]
Let $X$ be a direct summand of $W$, such that $[X,D']>[X,D]$. Then the exact sequence \[0\rightarrow U\rightarrow D\rightarrow V\rightarrow 0\] yields the existence of an exact sequence
\[0\rightarrow \Hom(X,U)\rightarrow \Hom(X,D)\rightarrow \Hom(X,V)\rightarrow \Ext^1(X,U)\rightarrow \Ext^1(X,D),\] such that the last map is not injective.\\[1ex]
 Thus, there exists a representation $Y$ and an exact sequence $0\rightarrow U\rightarrow Y\rightarrow X\rightarrow 0$, such that the pushout sequence splits and we obtain the commutative diagram 
\begin{displaymath}
% use packages: array
\begin{array}{lllllllllll}
 &  & 0 &  & ~~~~0 &  &  &  &  &  &  \\ 

 &  & \downarrow &  & ~~~\downarrow &  &  &  &  &  &  \\ 
0 & \rightarrow & U & \rightarrow & ~~~~Y & \rightarrow & X & \rightarrow & 0 &  &  \\ 
 &  & \downarrow &  & ~~~\downarrow &  & \downarrow &  &  &  &  \\ 
0 & \rightarrow & D & \rightarrow & D\oplus X & \rightarrow & X & \rightarrow & 0 &  &  \\  
&  & \downarrow &  & ~~~\downarrow &  &  &  &  &  &  \\ 
 &  & V & \rightarrow & ~~~~V &  &  &  &  &  &  \\ 
 &  & \downarrow &  & ~~~\downarrow &  &  &  &  &  &  \\ 
 &  & 0 &  & ~~~~0 &  &  &  &  &  &  \\ 
\end{array}
\end{displaymath}
with $D\oplus X<_{\deg}V\oplus Y<_{\deg}V\oplus U\oplus X$. We denote by $Z$ the representation that fulfills $W=X\oplus Z$ and obtain
\[D\oplus W<_{\deg}V\oplus Y\oplus Z<_{\deg}D'\oplus W,\] a contradiction.
A dual argument contradicts the assumption  $[D',X]>[D,X]$ for a direct summand $X$ of $W$.\\[1ex]
Assume $[X,D]=[X,D']$ and $[D,X]=[D',X]$ holds true for every direct summand $X$ of $W$. Then $\codim(D,D')=\codim(D\oplus W,D'\oplus W)$ and the Cancellation Theorem of \cite{Bo1} yields that the degeneration $D\oplus W<_{\deg}D'\oplus W$ is minimal if and only if $D<_{\deg}D'$ is.\\[1ex]
The only minimal, disjoint degeneration left is $D:=\U_{s,t}<_{\mdeg}\U_{t,s}=:D'$, where $s<t$.
The theorem then reads as follows: $D\oplus W<_{\deg}D'\oplus W$ is minimal if and only if every indecomposable direct summand $\V_k$ of $W$ fulfills $\delta_{s<k<t}=0$ and if every indecomposable direct summand $\U_{k,l}$ of $W$ fulfills $\delta_{k<t}\delta_{s<l<t}+\delta_{s<k<t}\delta_{t<l}=0$.\\[1ex]
If $s<k<t$, then the degeneration $\U_{t,s}\oplus \V_{k}<_{\mdeg}\U_{s,t}\oplus\V_{k}$ is not minimal since
\[\U_{t,s}\oplus\V_{k} <_{\deg} \U_{k,s}\oplus\V_t <_{\deg}\U_{s,t}\oplus\V_{k}\]
are proper degenerations.\\[1ex]
If $s\neq k< t$ and $s< l<t$  (or $s<k < t$ and $l>t$, respectively), then the degeneration $\U_{t,s}\oplus \U_{k,l}<_{\deg}\U_{s,t}\oplus\U_{k,l}$ is not minimal, since 
\[\U_{t,s}\oplus\U_{k,l} <_{\deg} \U_{k,s}\oplus\U_{t,l} <_{\deg}\U_{s,t}\oplus\U_{k,l}\]
\[(\U_{t,s}\oplus\U_{k,l} <_{\deg} \U_{k,s}\oplus\U_{t,l} <_{\deg}\U_{s,t}\oplus\U_{k,l},~ \textrm{respectively})\] 
are proper degenerations.\\[1ex]
Consider $W\in\rep^{\inj}_{K}(\Q_n,I_2)$, such that $M\coloneqq\U_{t,s}\oplus W<_{\deg}M'\coloneqq\U_{s,t}\oplus W$ in $\rep^{\inj}_{K}(\Q_n,I_2)(\dfs)$ and such that every direct summand of $W$ fulfills the assumptions.\\[1ex] If the degeneration $M<_{\deg}M'$ is not minimal, then there exists a representation $L$ fulfilling $M<_{\deg}L<_{\deg}M$. Without loss of generality, we can assume $M<_{\mdeg}L$.\\[1ex]
Then $[\V_k,M]\leq [\V_k,L]\leq [\V_k,M']$ for all $k$ and we can translate the statement as follows: The source vertices to the left of $s-1$ and to the right of $t$ coincide in $\olp(M)$, $\olp(L)$ and $\olp(M')$. Also, the number of arrows coincides in all three link patterns, since  $[\V_n,M]= [\V_n,L]= [\V_n,M']$.\\[1ex]
\textbf{Claim 1}:
 Let $\U_{k,l}$ be a direct summand of $M$, $L$ or $M'$. If $l<s$ or ($k<s $ and $l> t$) or ($k>t$ and $l>t$), then $\U_{k,l}$ is a direct summand of $M$, $L$ and $M'$.\\[1ex]
The proof of Claim 1 follows directly from Proposition \ref{combihom}.\\[1ex]
\textbf{Claim 2}:
Let $\U_{k,l}$ be a direct summand of $M$, $L$ or $M'$. If $t< k$ and $s< l<t$, then $\U_{k,l}$ is a direct summand of $M$, $L$ and $M'$.
\begin{proof}[Proof of Claim 2]
Let $t< k$ and $s< l<t$ for two integers $k$ and $l$.\\[1ex]
First, we assume that $U_{k,l}$ is a direct summand of $M$, but not a direct summand of $L$. 
 Since $M<_{\mdeg}L$, the indecomposable $\U_{k,l}$ must be changed by some minimal, disjoint part of the degeneration. The only possibilities for a change like that are the following:\\[1ex]
\textbf{1st case:} The indecomposable $\U_{k',l}$ is a direct summand of $L$, such that $k\neq k'$.\\[0.4ex]
1.1. The minimal, disjoint part is $\U_{k,l}\oplus\V_{k'}<_{\mdeg}\U_{k',l}\oplus \V_k$, such that $k'<k$:\\[0.2ex]
 The indecomposable $\V_{k'}$ can only be a direct summand of $M$ if $k'<s$ or $k'>t$.\\
 If $k'<s$, we obtain $[\U_{k',t},M]<[\U_{k',t},L]$ and if $k'>t$, we obtain $[\U_{k',l},M]<[\U_{k',l},L]$, a contradiction.\\[0.4ex]
1.2. The minimal, disjoint part is $\U_{k,l}\oplus\U_{k',l'}<_{\mdeg}\U_{k',l}\oplus \U_{k,l'}$, such that $k<k'$ and $l'<l$, or such that $k'<k$ and $l<l'$:\\[0.2ex]
 The indecomposable $\U_{k',l'}$ can only be a direct summand of $M$ if $k'>t$ or $l'<s$, or if $k'<s$ and $l'>t$. As has been shown in claim 1, every indecomposable $\U_{i,j}$ with $j<s$, or with $j>t$ and $i<s$ is either a direct summand of $M$, $L$ and $M'$ or a direct summand of none of them. Thus, $k'>t$ and if $k<k'$ and $l'<l$, we obtain $[\U_{k,l'},M]<[\U_{k,l'},L]$. If $k'<k$ and $l<l'$, we obtain $[\U_{k',l},M]<[\U_{k',l},L]$, a contradiction.\\[0.4ex]
1.3. The minimal, disjoint part is $\U_{k,l}\oplus\U_{l',k'}<_{\mdeg}\U_{l',k}\oplus \U_{k',l}$:\\[0.2ex]
 The indecomposable $\U_{l',k'}$ can only be a direct summand of $M$ if $l'>t$ or $k'<s$, or if $l'<s$ and $k'>t$. As has been shown in claim 1, every indecomposable $\U_{i,j}$ with $j<s$, or with $j>t$ and $i<s$ is either a direct summand of $M$, $L$ and $M'$ or a direct summand of none of them.\\[1ex] Thus, $l'>t$ and the only cases possible are $l<l'<k'<k$ and $l<k'<k<l'$. We immediately obtain $[\U_{k',l},M]<[\U_{k',l},L]$, a contradiction.\\[0.4ex]
\textbf{2nd case:} The indecomposable $\U_{k,l'}$ is a direct summand of $L$, such that $l\neq l'$.\\[0.4ex]
2.1. The minimal, disjoint part is $\U_{k,l}\oplus\V_{l'}<_{\mdeg}\U_{k,l'}\oplus \V_l$, such that $l<l'$:\\[0.2ex]
 The indecomposable $\V_{l'}$ can only be a direct summand of $M$ if $l'<s$ or $l'>t$.\\
 Thus, $l'>t$ and we obtain  $[\U_{t,l},M]<[\U_{t,l},L]$, a contradiction.\\[0.4ex]
2.2. The minimal, disjoint part is $\U_{k,l}\oplus\U_{l',k'}<_{\mdeg}\U_{k,l'}\oplus \U_{l,k'}$:\\[0.2ex]
 The indecomposable $\U_{l',k'}$ can only be a direct summand of $M$ if $l'>t$ or $k'<s$, or if $l'<s$ and $k'>t$. As has been shown in claim 1, every indecomposable $\U_{i,j}$ with $j<s$, or with $j>t$ and $i<s$ is either a direct summand of $M$, $L$ and $M'$ or a direct summand of none of them, thus, $l'>t$. But then we obtain $[\U_{t,l},M]<[\U_{t,l},L]$ , a contradiction.\\[0.4ex]
\textbf{3rd case:} The indecomposable $\U_{l,k}$ is a direct summand of $L$.\\[0.4ex]
Then $[\U_{1,t},M]<[\U_{1,t},L]$ if $s>1$ and $[\U_{t,n},M]<[\U_{t,n},L]$ if $t<n$. Of course, if $s=1$ and $t=n>2$, no representation $W$ as given in the assumption can exist at all, a contradiction.\\[0.4ex]
The assumption that $U_{k,l}$ is a direct summand of $L$, but not a direct summand of $M$ can be contradicted by a similar argumentation.\qedhere
\end{proof}
Claim 1 and Claim 2 show that all arrows $l\rightarrow k$ with $b_{k,l}(M)=b_{k,l}(M')$ and $k,l\notin\{s,t\}$ coincide in $\olp(M)$, $\olp(L)$ and $\olp(M')$. 
The minimal, disjoint piece of the degeneration $D\oplus W<_{\mdeg}L$, therefore, has to be one of the following three.
\begin{itemize}
 \item $\U_{t,s} <_{\mdeg}\U_{s,t}$: Then $L\cong M'$, a contradiction to the assumption $L<_{\deg} M'$.
\item $\U_{t,s}\oplus \V_{k'} <_{\mdeg}\U_{t,k'}\oplus \V_s$ with $k'>t$: In this case $\U_{t,k'}\oplus \V_s\nless_{\deg}\U_{s,t}\oplus \V_{k'}$ and therefore  $L\nless_{\deg} M'$, a contradiction.
\item $\U_{t,s}\oplus \V_{k'} <_{\mdeg}\U_{k',s}\oplus \V_t$ with $k'<s$: In this case $\U_{k',s}\oplus \V_t\nless_{\deg}\U_{s,t}\oplus \V_{k'}$  and therefore  $L\nless_{\deg} M'$, a contradiction.
\end{itemize}
Since we obtain a contradiction in each case, the degeneration $M<_{\deg}M'$ is minimal.
\end{proof}
Note that in the setup of Theorem \ref{indecgeneral}, the condition $[X,D]-[X,D']=0$ is sufficient in most cases. The only exceptions are the minimal, disjoint degenerations 
 $D=\U_{s,t}<_{\mdeg}\V_s\oplus \V_t=D'$, such that $s<t$, and $D=\U_{r,t}\oplus\V_s<_{\mdeg}\U_{s,t}\oplus\V_r=D'$, such that $s<r$.\\[1ex]
 The concrete minimal degenerations are obtained easily from Proposition \ref{dimhom}. Furthermore, each minimal degeneration is of codimension $1$ (which is, as well, clear from the theory of spherical varieties, see \cite{Br1}; a concrete proof is given in \cite{Tim}).

\subsubsection[Minimal singularities in B-orbit closures]{Minimal singularities in $B$-orbit closures}\label{singularities}
Since the bijection of Lemma \ref{bijection} preserves types of singularities, we consider singularities in $R_{\dfs}^{\inj}(Q_n,I_2)$ in order to examine singularities in the $B$-orbit closures in $\N_n^{(2)}$.\\[1ex]
We denote a representation in $\rep^{\inj}_{K}(\Q_n,I_2)$ by a capital letter and the corresponding point in $R_{\dfs}^{\inj}(Q,I)$ by the same small letter.\\[1ex]
In the following, minimal singularities are discussed, that is, given a minimal degeneration $M<_{\mdeg}M'$, we examine if $m'$ is a singularity in $\overline{\Orb_M}$, where $M\in \rep^{\inj}_{K}(\Q_n,I_2)(\dfs)$. 
Note that if a point $m'$ is contained in the singular locus, then every $\GL_{\dfs}$-conjugate of $m'$ is contained as well. Therefore, it suffices to consider representations in normal form.\\[1ex]
Given a minimal degeneration $M<_{\mdeg}M'$ in $\rep^{\inj}_{K}(\Q_n,I_2)(\dfs)$, we know that \mbox{$M=D\oplus W$} and $M'=D'\oplus W$, such that $D$ and $D'$ are disjoint and $D<_{\mdeg}D'$ is a minimal, disjoint degeneration  and $\codim(M,M')=\codim(D,D')=1$.\\[1ex]
Furthermore, $[X,D]=[X,D']$ and $[D,X]=[D',X]$ for every indecomposable direct summand $X$ of $W$. Of course, then $[X,M]=[X,M']$ and $[M,X]=[M',X]$ holds true as well.\\[1ex]
 The following theorem is due to K. Bongartz (see \cite{Bo1}) and yields the reduction to minimal, disjoint degenerations; we formulate it for the setup given above. 
\begin{theorem}\label{cancelsing}
 Let $D<_{\deg}D'$ and $M=D\oplus U<_{\mdeg}D'\oplus U=M'$ be degenerations of the same codimension.
 Then the two pointed varieties $(\overline{\Orb_{D\oplus U}},d'\oplus u)$ and $(\overline{\Orb_D},d')$ are (very) smoothly equivalent.
\end{theorem}
Thus, the pointed varieties $(\overline{\Orb_D},d')$ and $(\overline{\Orb_M},m')$ are (very) smoothly equivalent. Therefore, in order to classify the minimal singularities, it suffices to describe singularities arising from the minimal, disjoint degenerations in \cite[Theorem 4.6]{BoRe}.\\[1ex]
K. Bongartz  proves the following theorem (see \cite{Bo1}) which can easily be applied in the setup above.
\begin{theorem}\label{smoothext}
Let  $M<_{\mdeg}M'=U\oplus V$ be a minimal, disjoint degeneration of codimension one.
 Then $\overline{\Orb_M}$ is smooth at $m'$.
\end{theorem} 
\begin{corollary}\label{singext}
 For each minimal, disjoint degeneration $D<_{\mdeg}D'$ given in Proposition \ref{mindisj} by extensions, the point $d'$ is smooth in $\overline{\Orb_D}$.
\end{corollary}
 In case of the minimal, disjoint degeneration $\U_{t,s}<_{\mdeg}\U_{s,t}$ for $s<t$, the question about minimal singularities is still open - we start the discussion in the following. Let us define $V_i = \langle e_1,...,e_i\rangle$ to be the span of the first $i$ and $V_{\geq i}\coloneqq\langle e_i,\punkte,e_n\rangle$ of the last $n-i+1$ coordinate vectors of $K^n$. 
\begin{proposition}\label{descrequ}
Let $N\in\N^{(2)}$, then $B.N$ is given by matrices $X$ fulfilling the equations $X^{2}=0$ and \mbox{$\dim(X\cdot V_j\cap V_{\geq i})=\dim(N\cdot V_j\cap V_{\geq i})$} for all $i,j\in\{1,\punkte,n\}$.
\end{proposition}
\begin{proof}
The datum $\dim(N\cdot V_j\cap V_{\geq i})$ is $B$-invariant, it therefore suffices to consider matrices in normal form.
Given two matrices $N,N'$ in normal form, 
$\dim(N\cdot V_j\cap V_{\geq i})=\dim(N'\cdot V_j\cap V_{\geq i})$ holds true  for all $i,j\in\{1,\punkte,n\}$ if and only if $N=N'$.
\end{proof}
We denote by $E_{i,j}$ the $n\times n$-matrix given by $(E_{i,j})_{i,j}=1$ and $(E_{i,j})_{k,l}=0$ otherwise.
\begin{example}\label{singmin}
The point $E_{1,2}$ is smooth in the closure of $B.E_{2,1}\subseteq\N_2^{(2)}$:\\
It follows from \cite[Theorem 4.6]{BoRe} that $\N_2^{(2)}=\overline{B.E_{2,1}}=B.E_{2,1}\cup B.E_{1,2}\cup\{0\}$. Then due to Proposition \ref{descrequ}:
\begin{itemize}
 \item $B.E_{2,1}=\left\lbrace \left(\begin{array}{cc}
n_{1,1} & n_{1,2} \\ 
n_{2,1} & n_{2,2}
                       \end{array}\right)\mid~ n_{2,1}\neq 0;~ n_{1,1}+n_{2,2}=0;~ n_{1,1}n_{2,2}-n_{1,2}n_{2,1}=0\right\rbrace $
\item $B.E_{1,2}=\left\lbrace \left(\begin{array}{cc}
0 & n_{1,2} \\ 
0 & 0
                       \end{array}\right)\mid~ n_{1,2}\neq 0\right\rbrace $
\item $B.0=\left\lbrace \left(\begin{array}{ll}
0 & 0 \\ 
0 & 0
                       \end{array}\right)\right\rbrace$
\end{itemize}
The ideal 
$\langle n_{1,1}+n_{2,2},n_{1,1}n_{2,2}-n_{1,2}n_{2,1} \rangle\subset K[n_{1,1},n_{1,2},n_{2,1},n_{2,2}]$  is reduced, and we can read off the smoothness of every point contained in $\overline{B.E_{2,1}}$, except the zero-matrix, in the associated Jacobian matrix
\[J=\left( \begin{array}{llll}
1 & 0 & 0 & 1 \\ 
n_{2,2} & -n_{2,1} & -n_{1,2} & n_{1,1}
            \end{array}\right).\]
\end{example}
In the example $n=3$, minimal singularities arise. 
\begin{example}
The orbits can due to proposition \ref{descrequ} be described by equations as follows.
\begin{itemize}
\item $B.E_{2,1}=\left\lbrace \left(
\begin{array}{ccc}
n_{1,1} & n_{1,2} & n_{1,3} \\ 
n_{2,1} & n_{2,2} & n_{2,3} \\ 
0 & 0 & 0
                                \end{array}\right)\mid~ % use packages: array
\begin{array}{c}
n_{2,1}\neq 0;~ n_{1,1}n_{2,2}-n_{1,2}n_{2,1}=0; \\ 
n_{1,1}n_{2,3}-n_{1,3}n_{2,1}=0;~ n_{1,1}+n_{2,2}=0
                                                        \end{array}
\right\rbrace$
\item $\overline{B.E_{2,1}}=\left\lbrace \left(
\begin{array}{ccc}
n_{1,1} & n_{1,2} & n_{1,3} \\ 
n_{2,1} & n_{2,2} & n_{2,3} \\ 
0 & 0 & 0
                                \end{array}\right)\mid~ % use packages: array
\begin{array}{c}
n_{1,1}+n_{2,2}=0;~ n_{1,1}n_{2,2}-n_{1,2}n_{2,1}=0; \\ 
 n_{1,1}n_{2,3}-n_{1,3}n_{2,1}=0
                                                        \end{array}\right\rbrace $
\end{itemize}

By using the computer algebra system ``Singular'', we can show that the induced ideal \[I_{2,1}\coloneqq\left\langle n_{1,1}+n_{2,2},~ n_{1,1}n_{2,2}-n_{1,2}n_{2,1},~n_{1,1}n_{2,3}-n_{1,3}n_{2,1},~n_{3,1},~n_{3,2},~n_{3,3}\right\rangle\]
is reduced in $K[n_{1,1},n_{1,2},n_{1,3},n_{2,1},n_{2,2},n_{2,3},n_{3,1},n_{3,2},n_{3,3}]$.\\[1ex]
Thus, the associated Jacobian matrices can be computed directly. Without loss of generality, we consider the shortened ideal, deleting zero-variables.\\[1ex]
The associated Jacobian matrix is
\[J=\left( \begin{array}{cccccc}
n_{2,2} & -n_{2,1} & 0 & -n_{1,2} & n_{1,1} & 0 \\ 
n_{2,3} & 0 & n_{2,1} & -n_{1,3} & 0 & n_{1,1} \\ 
1 & 0 & 0 & 0 & 1 & 0
                     \end{array}\right),\]
and we directly see that $E_{1,2}$, $E_{1,3}$ and $E_{2,3}$ are singular points in $\overline{B.E_{2,1}}$.
\end{example}

\section[Maximal parabolic actions in nilpotency degree 3]{Maximal parabolic actions on $\N_n^{(3)}$}\label{degree3}
The only case not considered so far where the algebra associated to the action of $P$ on $\N_n^{(x)}$ is representation-finite comes up for $x=3$ and a maximal parabolic subgroup (that is, it is given by $2$ blocks) $P$ of arbitrary block-sizes $\df:=(b_1,b_2)$. We classify this case in the following before proving that it is the only finite case in Section \ref{fincrit}.
By Lemma \ref{bijection}, we need to consider representations of the algebra $\A\coloneqq K\Q_2/I_3$.

\begin{proposition}\label{indec3nilp}
 The indecomposable representations in $\rep_K(\Q_2,I_3)$ are (up to isomorphism) of the form 
\begin{center} \begin{tikzpicture}
\matrix (m) [matrix of math nodes, row sep=0.05em,
column sep=2em, text height=1.5ex, text depth=0.2ex]
{U=& K^i  & K^j \\};
\path[->]
(m-1-2) edge node[above=0.05cm] {$e_{i,j}$} (m-1-3)
(m-1-3) edge [loop right] node{$N$} (m-1-3);\end{tikzpicture}\end{center}  
for certain integers $i,j$ and nilpotent matrices $N$, where $e_{i,j}$ is the natural embedding. They are listed explicitely here; we thereby denote the dimension vectors as in the proof. 
\begin{tabular}[h]{|l|l|l||l|l|l|}
\hline
Name & 
Dim.                
& Matrix & Name & 
Dim.                
& Matrix\\ 
\hline
\begin{footnotesize}$U_{0,1}$\end{footnotesize} &
\begin{footnotesize}$0 1 $  \end{footnotesize}& 
\begin{footnotesize}$0$ \end{footnotesize}  &
\begin{footnotesize}$U_{1,4}$\end{footnotesize}& 
\begin{footnotesize}  $0 1  1 2 0 1 $ \end{footnotesize} & 
\begin{footnotesize}$E^{(4)}_{2,1}+E^{(4)}_{2,3} +E^{(4)}_{3,4}$ \end{footnotesize}  \\
\hline

\begin{footnotesize}$U_{1,1}$\end{footnotesize}& 
\begin{footnotesize}   $1 1 $   \end{footnotesize} & 
\begin{footnotesize}$0$  \end{footnotesize} &
\begin{footnotesize}$U^{(1)}_{2,4}$\end{footnotesize}& 
\begin{footnotesize}   $1 1  1 2 0 1$  \end{footnotesize} &  
\begin{footnotesize}$E^{(4)}_{3,1}+E^{(4)}_{4,2} +E^{(4)}_{4,3}$ \end{footnotesize}\\
\hline

\begin{footnotesize}$U_{1,0}$\end{footnotesize}& 
\begin{footnotesize}   $1 0 $ \end{footnotesize} &   
\begin{footnotesize}$0$ \end{footnotesize}&
\begin{footnotesize}$U^{(2)}_{2,4}$\end{footnotesize}& 
\begin{footnotesize}  $0 1  1 2 1 1  $ \end{footnotesize} & 
\begin{footnotesize}$E^{(4)}_{1,2}+E^{(4)}_{1,3} +E^{(4)}_{3,4}$ \end{footnotesize}  \\
\hline

\begin{footnotesize}$U_{0,2}$\end{footnotesize} & 
\begin{footnotesize}   $0 1 0 1 $  \end{footnotesize} & 
\begin{footnotesize}$E^{(2)}_{2,1}$\end{footnotesize}&
\begin{footnotesize}$U_{3,4}$\end{footnotesize}& 
\begin{footnotesize}   $1 1  1 2 1 1  $ \end{footnotesize} & 
\begin{footnotesize}$E^{(4)}_{1,2}+E^{(4)}_{1,4}+E^{(4)}_{4,3}$ \end{footnotesize} \\

\hline
\begin{footnotesize}$U^{(1)}_{1,2}$\end{footnotesize} & 
\begin{footnotesize}   $1 1  0 1 $  \end{footnotesize}  &
\begin{footnotesize}$E^{(2)}_{2,1}$\end{footnotesize}&
\begin{footnotesize}$U^{(1)}_{2,5}$\end{footnotesize}& 
\begin{footnotesize}   $1 2 1 2 0 1 $   \end{footnotesize}  & 
\begin{footnotesize}$E^{(5)}_{3,1}+E^{(5)}_{4,2} +E^{(5)}_{3,4}+E^{(5)}_{1,5} $\end{footnotesize} \\
\hline

\begin{footnotesize}$U^{(2)}_{1,2}$\end{footnotesize}& 
\begin{footnotesize}   $0 1  1 1 $  \end{footnotesize} &
\begin{footnotesize}$E^{(2)}_{1,2}$ \end{footnotesize} &
\begin{footnotesize}$U^{(2)}_{2,5}$\end{footnotesize}& 
\begin{footnotesize}   $0 1  1 2 1 2 $   \end{footnotesize}  & 
\begin{footnotesize}$E^{(5)}_{3,2}+E^{(5)}_{1,4} +E^{(5)}_{2,5}+E^{(5)}_{4,5} $\end{footnotesize} \\
\hline

\begin{footnotesize}$U_{2,2}$\end{footnotesize}& 
\begin{footnotesize}   $1 1  1 1 $ \end{footnotesize}  & 
\begin{footnotesize}$E^{(2)}_{2,1} $\end{footnotesize}&
\begin{footnotesize}$U^{(1)}_{3,5}$\end{footnotesize}& 
\begin{footnotesize}   $1 2  1 2 1 1 $  \end{footnotesize} & 
\begin{footnotesize}$E^{(5)}_{4,2}+E^{(5)}_{2,3}+E^{(5)}_{5,3}+E^{(5)}_{1,5}$  \end{footnotesize} \\
\hline

\begin{footnotesize}$U_{0,3}$\end{footnotesize}
& \begin{footnotesize}   $0 1  0 1 0 1 $  \end{footnotesize} &
\begin{footnotesize} $E^{(3)}_{2,1}+E^{(3)}_{3,2}$ \end{footnotesize}&
 \begin{footnotesize}$U^{(2)}_{3,5}$\end{footnotesize}& 
\begin{footnotesize}   $1 1  1 2 1 2 $   \end{footnotesize} & 
\begin{footnotesize}$E^{(5)}_{1,2}+E^{(5)}_{4,3} +E^{(5)}_{1,4}+E^{(5)}_{2,5}$ \end{footnotesize} \\
\hline

\begin{footnotesize}$U^{(1)}_{1,3}$\end{footnotesize}& 
\begin{footnotesize}   $1 1  0 1 0 1 $   \end{footnotesize} &
\begin{footnotesize}$E^{(3)}_{2,1}+E^{(3)}_{3,2}$ \end{footnotesize}&
\begin{footnotesize}$U_{2,6}$\end{footnotesize} & 
\begin{footnotesize}   $0 1  1 2 1 2 0 1  $  \end{footnotesize} & 
\begin{footnotesize} $E^{(6)}_{1,3}+E^{(6)}_{2,1}+E^{(6)}_{2,4}-E^{(6)}_{4,3} +E^{(6)}_{5,1}+E^{(6)}_{6,2}$\end{footnotesize}
\\
\hline

\begin{footnotesize}$U^{(1)}_{2,3}$\end{footnotesize}& 
\begin{footnotesize}   $1 1 1 1 0 1 $  \end{footnotesize} &
\begin{footnotesize}$E^{(3)}_{2,1}+E^{(3)}_{3,2}$  \end{footnotesize} &
 \begin{footnotesize}$U_{3,6} $\end{footnotesize}& 
\begin{footnotesize}  $ 1 2  1 2 1 2 $ \end{footnotesize} & 
\begin{footnotesize} $E^{(6)}_{1,2}+E^{(6)}_{4,2} +E^{(6)}_{5,3}+E^{(6)}_{4,5} +E^{(6)}_{2,6}-E^{(6)}_{5,6}$\end{footnotesize}
\\
\hline

\begin{footnotesize}$U^{(2)}_{1,3}$\end{footnotesize} & 
\begin{footnotesize}   $0 1 1 1 0 1$  \end{footnotesize} & 
\begin{footnotesize}$E^{(3)}_{3,1}+E^{(3)}_{1,2}$\end{footnotesize}&
\begin{footnotesize}$U^{(1)}_{3,6}$\end{footnotesize}& 
\begin{footnotesize}   $1 1  1 2 1 2 0 1$          \end{footnotesize} & 
\begin{footnotesize}$E^{(6)}_{2,1}+E^{(6)}_{3,2} +E^{(6)}_{3,4}-E^{(6)}_{4,1} +E^{(6)}_{5,2}+E^{(6)}_{6,3}$\end{footnotesize}
\\
\hline

\begin{footnotesize}$U^{(2)}_{2,3}$\end{footnotesize} & 
\begin{footnotesize}   $1 1  0 1 1 1$  \end{footnotesize} & 
\begin{footnotesize}$E^{(3)}_{3,1}+E^{(3)}_{2,3}$  \end{footnotesize}&
\begin{footnotesize}$U^{(2)}_{3,6} $\end{footnotesize} & 
\begin{footnotesize}   $0 1  1 2 1 2 1 1 $  \end{footnotesize} & 
\begin{footnotesize}$E^{(6)}_{2,5}+E^{(6)}_{3,2} +E^{(6)}_{3,6}+E^{(6)}_{5,4}+E^{(6)}_{6,1}-E^{(6)}_{6,5}$\end{footnotesize}
\\
\hline

\begin{footnotesize}$U_{3,3}$\end{footnotesize}& 
\begin{footnotesize}   $1 1 1 1 1 1 $ \end{footnotesize} &
\begin{footnotesize}$E^{(3)}_{2,1}+E^{(3)}_{3,2}$  \end{footnotesize}&
\begin{footnotesize}$U_{4,6}$\end{footnotesize}& 
\begin{footnotesize}   $1 1  1 2 1 2 1 1 $   \end{footnotesize} & 
\begin{footnotesize} $E^{(6)}_{2,1}+E^{(6)}_{3,2} +E^{(6)}_{3,5}+E^{(6)}_{4,3}-E^{(6)}_{5,1}+E^{(6)}_{6,2}$\end{footnotesize}
\\
\hline

\begin{footnotesize}$U^{(3)}_{1,3}$\end{footnotesize} &
\begin{footnotesize}    $0 1  0 1 1 1  $ \end{footnotesize} & 
\begin{footnotesize}$E^{(3)}_{1,2}+E^{(3)}_{2,3}$\end{footnotesize} & 
\begin{footnotesize}$U_{3,7} $ \end{footnotesize}& 
\begin{footnotesize}   $1 2  1 3 1 2 $  \end{footnotesize} &
\begin{footnotesize} $E^{(7)}_{1,2}+E^{(7)}_{4,2} -E^{(7)}_{6,3}-E^{(7)}_{1,5} -E^{(7)}_{4,6}-E^{(7)}_{5,7}$\end{footnotesize}
\\
\hline

\begin{footnotesize}$U^{(3)}_{2,3}$\end{footnotesize}&
\begin{footnotesize}   $0 1 1 1 1 1  $  \end{footnotesize} & 
\begin{footnotesize}$E^{(3)}_{1,2}+E^{(3)}_{2,3}$  \end{footnotesize}&
\begin{footnotesize}$U_{4,7} $  \end{footnotesize}& 
\begin{footnotesize}   $1 2  2 3 1 2 $ \end{footnotesize} &
\begin{footnotesize} $E^{(7)}_{5,2}+E^{(7)}_{1,3} +E^{(7)}_{2,4}+E^{(7)}_{6,4}+E^{(7)}_{3,7}+E^{(7)}_{6,7} $ \end{footnotesize} \\
\hline

\end{tabular}

\end{proposition}
\begin{proof}
In order to calculate the representatives of the isomorphism classes of indecomposable representations,  
we make use of covering theory and calculate the Auslander-Reiten quiver of
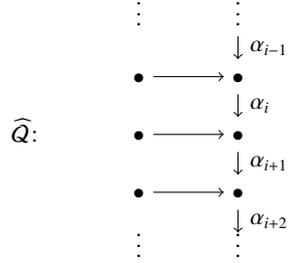
\begin{center}{\small\begin{tikzpicture}
\matrix (m) [matrix of math nodes, row sep=1em,
column sep=3em, text height=1.5ex, text depth=0.2ex]
{ & \vdots & \vdots \\
 & \bullet  & \bullet \\ 
\widehat{\Q}\colon & \bullet  & \bullet \\
 & \bullet  & \bullet \\
 & \vdots &\vdots  \\};
\path[->]
(m-2-2) edge  (m-2-3)
(m-3-2) edge   (m-3-3) 
(m-4-2) edge  (m-4-3)
(m-1-3) edge node[right=0.05cm] {$\alpha_{i-1}$} (m-2-3) 
(m-2-3) edge node[right=0.05cm] {$\alpha_i$} (m-3-3)
(m-3-3) edge node[right=0.05cm] {$\alpha_{i+1}$} (m-4-3)
(m-4-3) edge node[right=0.05cm] {$\alpha_{i+2}$} (m-5-3);\end{tikzpicture}}\end{center}
together with the induced ideal $\widehat{I}$, generated by all paths $\alpha_{i+1}\alpha_i\alpha_{i-1}$. The natural free action of the group $\mathbf{Z}$ on $\widehat{\Q}$ is given by shifting the rows. Due to covering theory (see \cite{BoGa} and \cite{Ga3}), there is a bijection between the indecomposables in $\A$ and the indecomposables in $\hat{\A}/\mathbf{Z}$. For every integer $k$, we consider the finite subquiver
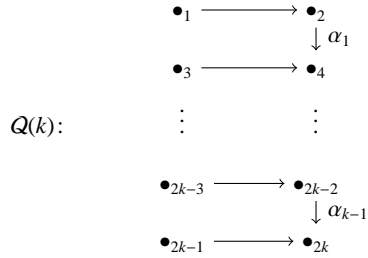
\begin{center}{\small\begin{tikzpicture}
\matrix (m) [matrix of math nodes, row sep=1em,
column sep=3em, text height=1.5ex, text depth=0.2ex]
{ & \bullet_1  & \bullet_2 \\ 
 & \bullet_3  & \bullet_4 \\
\Q(k)\colon & \vdots  & \vdots \\
 & \bullet_{2k-3}  & \bullet_{2k-2} \\
 & \bullet_{2k-1}  & \bullet_{2k} \\};
\path[->]
(m-1-2) edge  (m-1-3)
(m-2-2) edge   (m-2-3)
(m-5-2) edge  (m-5-3)
(m-4-2) edge  (m-4-3)
(m-1-3) edge node[right=0.05cm] {$\alpha_{1}$} (m-2-3)
(m-4-3) edge node[right=0.05cm] {$\alpha_{k-1}$} (m-5-3);\end{tikzpicture}}\end{center}
together with the ideal $I(k)$ generated by the paths $\alpha_{i+1}\alpha_i\alpha_{i-1}$ for $i\in\{2,\punkte,k-2\}$.
 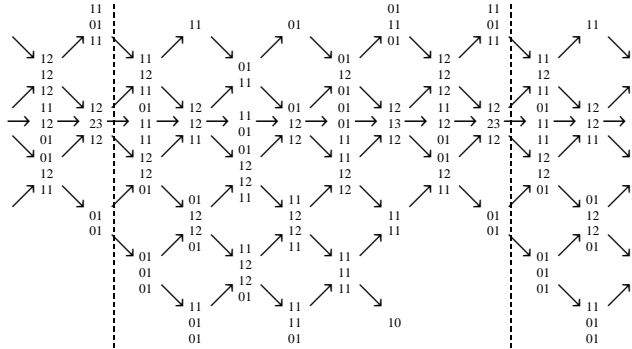
\begin{figure}[b]
\setlength{\unitlength}{0.55mm}
\hspace{1.5cm}\begin{picture}(30,20)(-8,6)
\multiput(21,-65)(0,2){42}{\line(0,1){1}}
\multiput(117,-65)(0,2){42}{\line(0,1){1}}
 \put(-4,6){$\searrow$}
\put(-4,-6){$\nearrow$}\put(-4,-18){$\searrow$}\put(-5,-12){$\rightarrow$}
\put(-4,-30){$\nearrow$}
  \put(0,0){\begin{tiny}$\begin{array}{l}
1 2 \\ 
1 2 \\ 
1 2  \end{array}$\end{tiny}}\put(8,6){$\nearrow$}\put(8,-6){$\searrow$}
  \put(0,-12){\begin{tiny}$\begin{array}{l}
1 1 \\ 
1 2 \\ 
0 1  \end{array}$\end{tiny}}\put(7,-12){$\rightarrow$}
  \put(0,-24){\begin{tiny}$\begin{array}{l}
0 1 \\ 
1 2 \\ 
1 1  \end{array}$\end{tiny}}\put(8,-18){$\nearrow$}\put(8,-30){$\searrow$}
  \put(12,12){\begin{tiny}$\begin{array}{l}
1 1 \\ 
0 1 \\ 
1 1  \end{array}$\end{tiny}}
  \put(12,-12){\begin{tiny}$\begin{array}{l}
1 2 \\ 
2 3 \\ 
1 2  \end{array}$\end{tiny}}
  \put(12,-36){\begin{tiny}$\begin{array}{l}
0 1 \\ 
0 1  \end{array}$\end{tiny}}
  %%%%%%%%%%%%%%%%%%%%%%%%
  \put(20,6){$\searrow$}

 \put(20,-6){$\nearrow$}\put(20,-18){$\searrow$}\put(19,-12){$\rightarrow$}

  \put(20,-30){$\nearrow$}\put(20,-42){$\searrow$}
  %%%%%%%%%%%%%%%%%%%%%%%%
    \put(24,0){\begin{tiny}$\begin{array}{l}
1 1 \\ 
1 2 \\ 
1 1  \end{array}$\end{tiny}}\put(32,6){$\nearrow$}\put(32,-6){$\searrow$}
  \put(24,-12){\begin{tiny}$\begin{array}{l}
0 1 \\ 
1 1 \\ 
1 1  \end{array}$\end{tiny}}\put(31,-12){$\rightarrow$}
  \put(24,-24){\begin{tiny}$\begin{array}{l}
1 2 \\ 
1 2 \\ 
0 1  \end{array}$\end{tiny}}\put(32,-18){$\nearrow$}\put(32,-30){$\searrow$}
  \put(24,-48){\begin{tiny}$\begin{array}{l}
0 1 \\ 
0 1 \\ 
0 1  \end{array}$\end{tiny}}\put(32,-42){$\nearrow$}\put(32,-54){$\searrow$}
  %%%%%%%%%%%%%%%%%%%%%%%%
   \put(36,12){\begin{tiny}$\begin{array}{l}
1 1   \end{array}$\end{tiny}}\put(44,6){$\searrow$}

  \put(36,-12){\begin{tiny}$\begin{array}{l}
1 2 \\ 
1 2 \\ 
1 1  \end{array}$\end{tiny}}\put(44,-6){$\nearrow$}\put(44,-18){$\searrow$}\put(43,-12){$\rightarrow$}

  \put(36,-36){\begin{tiny}$\begin{array}{l}
0 1 \\ 
1 2 \\ 
1 2 \\
0 1 \end{array}$\end{tiny}}\put(44,-30){$\nearrow$}\put(44,-42){$\searrow$}
  \put(36,-60){\begin{tiny}$\begin{array}{l}
1 1 \\ 
0 1 \\ 
0 1  \end{array}$\end{tiny}}\put(44,-54){$\nearrow$}
  %%%%%%%%%%%%%%%%%%%%%%%%
  \put(48,0){\begin{tiny}$\begin{array}{l}
0 1 \\ 
1 1   \end{array}$\end{tiny}}\put(56,6){$\nearrow$}\put(56,-6){$\searrow$}
  \put(48,-12){\begin{tiny}$\begin{array}{l}
1 1 \\ 
0 1   \end{array}$\end{tiny}}\put(55,-12){$\rightarrow$}
  \put(48,-24){\begin{tiny}$\begin{array}{l}
0 1 \\ 
1 2 \\ 
1 2 \\
1 1 \end{array}$\end{tiny}}\put(56,-18){$\nearrow$}\put(56,-30){$\searrow$}

  \put(48,-48){\begin{tiny}$\begin{array}{l}
1 1 \\ 
1 2 \\ 
1 2 \\
0 1 \end{array}$\end{tiny}}\put(56,-42){$\nearrow$}\put(56,-54){$\searrow$}
  %%%%%%%%%%%%%%%%%%%%%%%%
  \put(60,12){\begin{tiny}$\begin{array}{l}
0 1  \end{array}$\end{tiny}}\put(68,6){$\searrow$}

  \put(60,-12){\begin{tiny}$\begin{array}{l}
0 1 \\ 
1 2 \\ 
1 2  \end{array}$\end{tiny}}\put(68,-6){$\nearrow$}\put(68,-18){$\searrow$}\put(67,-12){$\rightarrow$}

\put(60,-36){\begin{tiny}$\begin{array}{l}
1 1 \\ 
1 2 \\ 
1 2 \\
1 1 \end{array}$\end{tiny}}\put(68,-30){$\nearrow$}\put(68,-42){$\searrow$}
  \put(60,-60){\begin{tiny}$\begin{array}{l}
1 1 \\ 
1 1 \\ 
0 1  \end{array}$\end{tiny}}\put(68,-54){$\nearrow$}
  %%%%%%%%%%%%%%%%%%%%%%%%
  \put(72,0){\begin{tiny}$\begin{array}{l}
0 1 \\ 
1 2 \\ 
0 1  \end{array}$\end{tiny}}\put(80,6){$\nearrow$}\put(80,-6){$\searrow$}
  \put(72,-12){\begin{tiny}$\begin{array}{l}
0 1 \\ 
0 1 \\ 
1 1  \end{array}$\end{tiny}}\put(79,-12){$\rightarrow$}
  \put(72,-24){\begin{tiny}$\begin{array}{l}
1 1 \\ 
1 2 \\ 
1 2  \end{array}$\end{tiny}}\put(80,-18){$\nearrow$}\put(80,-30){$\searrow$}

  \put(72,-48){\begin{tiny}$\begin{array}{l}
1 1 \\ 
1 1 \\ 
1 1  \end{array}$\end{tiny}}\put(80,-42){$\nearrow$}\put(80,-54){$\searrow$}
  %%%%%%%%%%%%%%%%%%%%%%%%
   \put(84,12){\begin{tiny}$\begin{array}{l}
0 1 \\ 
1 1 \\ 
0 1  \end{array}$\end{tiny}}\put(92,6){$\searrow$}

  \put(84,-12){\begin{tiny}$\begin{array}{l}
1 2 \\ 
1 3 \\ 
1 2  \end{array}$\end{tiny}}\put(92,-6){$\nearrow$}\put(92,-18){$\searrow$}\put(91,-12){$\rightarrow$}

\put(84,-36){\begin{tiny}$\begin{array}{l}
1 1 \\ 
1 1  \end{array}$\end{tiny}}\put(92,-30){$\nearrow$}
  \put(84,-60){\begin{tiny}$\begin{array}{l}
1 0   \end{array}$\end{tiny}}
 %%%%%%%%%%%%%%%%%%%%%%%%

  \put(96,0){\begin{tiny}$\begin{array}{l}
1 2 \\ 
1 2 \\ 
1 2  \end{array}$\end{tiny}}\put(104,6){$\nearrow$}\put(104,-6){$\searrow$}
  \put(96,-12){\begin{tiny}$\begin{array}{l}
1 1 \\ 
1 2 \\ 
0 1  \end{array}$\end{tiny}}\put(103,-12){$\rightarrow$}
  \put(96,-24){\begin{tiny}$\begin{array}{l}
0 1 \\ 
1 2 \\ 
1 1  \end{array}$\end{tiny}}\put(104,-18){$\nearrow$}\put(104,-30){$\searrow$}

 %%%%%%%%%%%%%%%%%%%%%%%%
   \put(108,12){\begin{tiny}$\begin{array}{l}
1 1 \\ 
0 1 \\ 
1 1  \end{array}$\end{tiny}}\put(116,6){$\searrow$}

  \put(108,-12){\begin{tiny}$\begin{array}{l}
1 2 \\ 
2 3 \\
1 2  \end{array}$\end{tiny}}\put(116,-6){$\nearrow$}\put(116,-18){$\searrow$}\put(115,-12){$\rightarrow$}

\put(108,-36){\begin{tiny}$\begin{array}{l}
0 1 \\ 
0 1  \end{array}$\end{tiny}}\put(116,-30){$\nearrow$}\put(116,-42){$\searrow$}
 %%%%%%%%%%%%%%%%%%%%%%%%
  
  \put(120,0){\begin{tiny}$\begin{array}{l}
1 1 \\ 
1 2 \\ 
1 1  \end{array}$\end{tiny}}\put(128,6){$\nearrow$}\put(128,-6){$\searrow$}
  \put(120,-12){\begin{tiny}$\begin{array}{l}
0 1 \\ 
1 1 \\ 
1 1  \end{array}$\end{tiny}}\put(127,-12){$\rightarrow$}
  \put(120,-24){\begin{tiny}$\begin{array}{l}
1 2 \\ 
1 2 \\ 
0 1  \end{array}$\end{tiny}}\put(128,-18){$\nearrow$}\put(128,-30){$\searrow$}

  \put(120,-48){\begin{tiny}$\begin{array}{l}
0 1 \\ 
0 1 \\ 
0 1  \end{array}$\end{tiny}}\put(128,-42){$\nearrow$}\put(128,-54){$\searrow$}

 %%%%%%%%%%%%%%%%%%%%%%%%
  \put(132,12){\begin{tiny}$\begin{array}{l}
1 1  \end{array}$\end{tiny}}\put(140,6){$\searrow$}

  \put(132,-12){\begin{tiny}$\begin{array}{l}
1 2 \\ 
1 2 \\ 
1 1  \end{array}$\end{tiny}}\put(140,-6){$\nearrow$}\put(140,-18){$\searrow$}\put(139,-12){$\rightarrow$}

  \put(132,-36){\begin{tiny}$\begin{array}{l}
0 1 \\ 
1 2 \\ 
1 2 \\
0 1  \end{array}$\end{tiny}}\put(140,-30){$\nearrow$}\put(140,-42){$\searrow$}
  \put(132,-60){\begin{tiny}$\begin{array}{l}
1 1 \\ 
0 1 \\ 
0 1  \end{array}$\end{tiny}}\put(140,-54){$\nearrow$}
 %%%%%%%%%%%%%%%%%%%%%%%% 
\end{picture}\rule{3mm}{0mm}\vspace{3.8cm}
\caption{The Auslander-Reiten quiver $\Gamma(\Q,I)$}\label{fig1}
\end{figure}

By calculating the Auslander-Reiten quivers $\Gamma(\Q(4),I(4))$ and $\Gamma(\Q(5),I(5))$ with elementary methods (see \cite[IV.4]{ASS}), we realize that all isomorphism classes of indecomposables in $K\Q(5)/I(5)$ already appear (up to the action of $\textbf{Z}$) in the  quiver $\Gamma(\Q(4),I(4))$. The translation of the indecomposables between the algebras is deduced directly from the action of $\textbf{Z}$.\\[1ex]
 It, therefore, suffices to calculate the indecomposable representations of the quiver $\Q(4)$
with the associated ideal $I(4)$ generated by the path $\alpha_{3}\alpha_{2}\alpha_1$. Figure \ref{fig1} shows $\Gamma(\Q,I)$, the dotted lines mark the mentioned identifications. We denote the indecomposables by their dimension vectors and directly delete zero rows in these, such that the identifications by the action of $\textbf{Z}$ can be seen right away.

The representations given in the table are all indecomposable, which can, for example, be proved by showing that the corresponding endomorphism rings are local. Either the representations have been considered in the $2$-nilpotent case or the number of given representations coincides with the number of corresponding indecomposables with the same dimension vectors in the Auslander Reiten quiver $\Gamma(\Q,I)$.
\end{proof}
Following Lemma \ref{bijection}, the $P$-orbits of $3$-nilpotent matrices are in bijection to the isomorphism classes of representations in $\rep^{\inj}_K(\Q_2,I_3)$ of dimension vector $\underline{d}$. The orbits are represented by direct sums of the indecomposable representations of Proposition \ref{indec3nilp} and by their translations to matrices, respectively.\\[1ex]
As in Section \ref{closures}, the orbit closures can be calculated by considering the dimensions of certain homomorphism spaces. The concrete dimension table is shown in Figure \ref{fig2}.

\subsubsection{The open orbit} 
We denote the matrix in normal form in the open $P$-orbit in $\N_n^{(3)}$ by $N_{\textrm{open}}$ and the representation in normal form  in the open $\GL_{\underline{d}}$-orbit in $R_{\underline{d}}^{\inj}(\Q_2,I_3)$ by $M_{\textrm{open}}$.
\begin{proposition}
The open orbit is represented by
\begin{enumerate} 
 \item[1.1] If $b_1\leq b_2$, such that $b_1\leq r$, then
 \[M_{\textrm{open}}=(\U_{1,3}^{(1)})^{b_1}\oplus (\U_{0,3})^{r-b_1} \oplus \left\lbrace % use packages: array
\begin{array}{ll}
0, &\textrm{if}~n=3r; \\ 
\U_{0,1}, & \textrm{if}~n=3r +1;  \\ 
\U_{0,2}, &  \textrm{if}~n=3r +2;
                                                                     \end{array}
  \right.\]
\item[1.2] If  $b_1\leq b_2$, such that $b_1>r$, then
 \[M_{\textrm{open}}=(\U_{3,6}^{(1)})^{b_1-r-1}\oplus(\U_{1,3}^{(1)})^{n-2b_1} \oplus \left\lbrace % use packages: array
\begin{array}{ll}
\U^{(1)}_{3,6}, & \textrm{if}~n=3r; \\ 
\U^{(1)}_{2,4}, & \textrm{if}~n=3r +1;  \\ 
\U^{(1)}_{1,2}, &  \textrm{if}~n=3r +2;
                                                                     \end{array}
  \right.\]
\item[2.1] If  $b_1\geq b_2$, such that $b_2\leq r$, then
\[M_{\textrm{open}}=(\U_{2,3}^{(1)})^{b_2}\oplus (\U_{3,3})^{r-b_2} \oplus \left\lbrace % use packages: array
\begin{array}{ll}
0, & \textrm{if}~n=3r; \\ 
\U_{1,1}, & \textrm{if}~n=3r +1;  \\ 
\U_{2,2}, &  \textrm{if}~n=3r +2;
                                                                     \end{array}
  \right.\]
\item[2.2] If $b_1\geq b_2$, such that $b_2>r$, then
\[M_{\textrm{open}}=(\U_{3,6}^{(1)})^{b_2-r-1}\oplus(\U_{2,3}^{(1)})^{n-2b_2} \oplus \left\lbrace % use packages: array
\begin{array}{ll}
\U^{(1)}_{3,6}, & \textrm{if}~n=3r; \\ 
\U^{(1)}_{2,4}, & \textrm{if}~n=3r +1;  \\ 
\U^{(1)}_{1,2}, &  \textrm{if}~n=3r +2;
                                                                     \end{array}
  \right.\]

\end{enumerate}
\end{proposition}
\begin{proof}
Let $n=3r$. In both cases 
$[M_{\textrm{open}},M_{\textrm{open}}] = 3r^2 -b_1b_2$. 
Since the open orbit is the orbit of maximal dimension, we have $\dim P.N_{\textrm{open}}=\dim\N^{(3)}$ and $\dim\Orb_{M_{\textrm{open}}}=\dim R_{\underline{d}}^{\inj}(\Q_2,I_3)$.
the claim follows from
 \[\dim\N^{(3)}= \dim P.N_{\textrm{open}}= \dim\Orb_{M_{\textrm{open}}} - n\cdot b_1 =b_1^2+n^2 -[M_{\textrm{open}},M_{\textrm{open}}]- n\cdot b_1.\]
The remaining cases can be shown analogously.
\end{proof}
The example of the action of the parabolic subgroup of block sizes $(2,2)$ can be found in Figure \ref{fig3}.
\begin{figure}[ht]\tiny\begin{center}\begin{tikzpicture}[xscale=0.67,yscale=0.58]
\node (M1) at (-1,0) {\fbox{$U^{(1)}_{2,4}~$ \vline\begin{tiny}$\begin{tabular}{c}
0 0 0 0 \\ 
0 0 0 0 \\ 
1 0 0 0 \\ 
0 1 1 0
        \end{tabular}$\end{tiny}}};
\node (M3) at (-4,-2) {\fbox{$U^{(1)}_{2,3}\oplus\U_{0,1}~$\vline\begin{tiny}$ \begin{tabular}{c}
0 0 0 0 \\ 
0 0 0 0 \\ 
0 1 0 0 \\ 
0 0 1 0
        \end{tabular}$\end{tiny}}};
\node (M8) at (2,-2) {\fbox{\begin{tiny}$\begin{tabular}{c}
0 0 0 0 \\ 
0 0 0 0 \\ 
0 1 0 0\\ 
0 0 1 0
        \end{tabular}$\end{tiny}\vline$~U^{(1)}_{1,3}\oplus\U_{1,1}$} };
\node (M9) at (-4,-4) { \fbox{$U^{(2)}_{1,3}\oplus\U_{1,1}~$\vline\begin{tiny}$ \begin{tabular}{c}
0 0 0 0 \\ 
0 0 1 0 \\ 
0 0 0 0 \\ 
0 1 0 0
        \end{tabular}$\end{tiny}} };
\node (M4) at (2,-4) {\fbox{\begin{tiny}$\begin{tabular}{c}
0 0 0 0 \\ 
0 0 0 0 \\ 
0 0 0 1 \\ 
0 1 0 0
        \end{tabular}$\end{tiny}\vline$~U^{(2)}_{2,3}\oplus\U_{0,1}$}};
\node (M2) at (-6,-6) { \fbox{$U^{(2)}_{2,4}~$\vline\begin{tiny}$ \begin{tabular}{c}
0 1 1 0 \\ 
0 0 0 0 \\ 
0 0 0 1 \\ 
0 0 0 0
        \end{tabular}$\end{tiny}}};
\node (M11) at (-1,-6) { \fbox{$U^{(1)}_{1,2}\oplus \U^{(2)}_{1,2}~$\vline\begin{tiny}$ \begin{tabular}{c}
0 0 0 0 \\ 
0 0 0 1 \\ 
1 0 0 0 \\ 
0 0 0 0
        \end{tabular}$\end{tiny}}};
\node (M5) at (-6.6,-8) {\fbox{$U^{(3)}_{2,3}\oplus\U_{0,1}~$\vline\begin{tiny}$ \begin{tabular}{c}
0 0 0 0 \\ 
0 0 1 0 \\ 
0 0 0 1 \\ 
0 0 0 0
        \end{tabular}$\end{tiny}}};
\node (M10) at (-2.9,-8) {\fbox{$U^{(3)}_{1,3}\oplus\U_{1,1}~$ \vline\begin{tiny}$ \begin{tabular}{c}
0 0 0 0 \\ 
0 0 1 0 \\ 
0 0 0 1 \\ 
0 0 0 0
        \end{tabular}$\end{tiny}}};
\node (M12) at (0.9,-8) {\fbox{\begin{tiny}$\begin{tabular}{c}
0 0 0 0 \\ 
0 0 0 0 \\ 
0 0 0 0 \\ 
0 1 0 0
        \end{tabular}$\end{tiny}\vline$~U^{(1)}_{1,2}\oplus \U_{1,1}\oplus \U_{0,1}$}};
\node (M6) at (4.6,-8) {\fbox{\begin{tiny}$\begin{tabular}{c}
0 0 0 0 \\ 
1 0 0 0 \\ 
0 0 0 0 \\ 
0 0 1 0
        \end{tabular}$\end{tiny}\vline$~U_{2,2}\oplus\U_{0,2}$}};
\node (M7) at (3,-10) {\fbox{\begin{tiny}$\begin{tabular}{c}
0 0 0 0 \\ 
0 0 0 0 \\ 
0 0 0 0 \\ 
0 0 1 0
        \end{tabular}$\end{tiny}\vline$~U_{2,2}\oplus\U^2_{0,1}$}};
\node (M13) at (-1,-12) {\fbox{$U^{(2)}_{1,2}\oplus \U_{1,1}\oplus \U_{0,1}~$\vline\begin{tiny}$ \begin{tabular}{c}
0 0 0 0 \\ 
0 0 0 1 \\ 
0 0 0 0 \\ 
0 0 0 0 
        \end{tabular}$\end{tiny}}};
\node (M14) at (-1,-14) {\fbox{$U^{2}_{1,1}\oplus \U^2_{0,1}~$\vline\begin{tiny}$ \begin{tabular}{c}
0 0 0 0 \\ 
0 0 0 0 \\
0 0 0 0 \\ 
0 0 0 0 
        \end{tabular}$\end{tiny}}};
\draw 
(M1) edge[-] (M3)
(M1) edge[-] (M8)
(M3) edge[-] (M9)
(M3) edge[-] (M4)
(M8) edge[-] (M9)
(M8) edge[-] (M4)
(M9) edge[-] (M2)
(M9) edge[-] (M10)
(M9) edge[-] (M11)
(M4) edge[-] (M6)
(M4) edge[-] (M12)
(M4) edge[-] (M11)
(M2) edge[-] (M5)
(M11) edge[-] (M12)
(M4) edge[-] (M11)
(M6) edge[-] (M7)
(M5) edge[-] (M13)
(M10) edge[-] (M13)
(M12) edge[-] (M13)
(M7) edge[-] (M13)
(M13) edge[-] (M14);
  \end{tikzpicture}\caption{Parabolic~ subgroup~ of~ block~ sizes~ $(2,2)$}\label{fig3}\end{center}\end{figure}
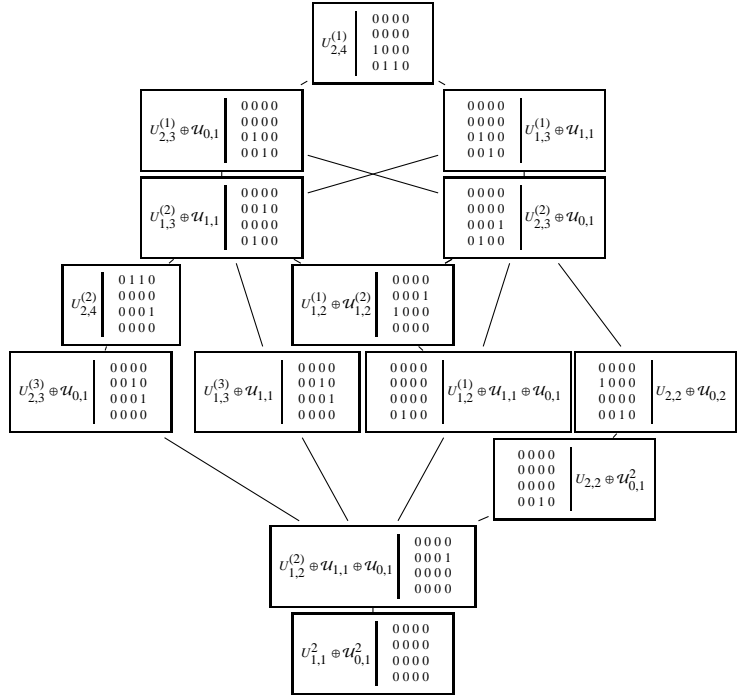

\begin{landscape}
\begin{figure}[ht]
\small
\begin{longtable}[c]{|p{0.3cm}|p{0.15cm}|p{0.15cm}|p{0.15cm}|p{0.15cm}|p{0.15cm}|p{0.2cm}|p{0.2cm}|p{0.2cm}|p{0.2cm}|p{0.2cm}|p{0.15cm}|p{0.15cm}|p{0.2cm}|p{0.2cm}|p{0.2cm}|p{0.2cm}|p{0.2cm}|p{0.2cm}|p{0.2cm}|p{0.15cm}|p{0.15cm}|p{0.15cm}|p{0.2cm}|p{0.2cm}|p{0.15cm}|p{0.2cm}|p{0.2cm}|p{0.15cm}|p{0.15cm}|p{0.15cm}|}
\hline
\begin{tiny}$\nearrow$\end{tiny} &\begin{tiny}$U_{0,1}$\end{tiny}& \begin{tiny}$U_{0,2}$\end{tiny} & \begin{tiny}$U_{0,3}$\end{tiny} & \begin{tiny}$U_{1,0}$\end{tiny} & \begin{tiny}$U_{1,1}$  \end{tiny}& \begin{tiny}$U^{(1)}_{1,2}$ \end{tiny} &\begin{tiny}$U^{(2)}_{1,2}$   \end{tiny}& \begin{tiny}$U^{(1)}_{1,3}$ \end{tiny} &\begin{tiny}$U^{(2)}_{1,3}$ \end{tiny}& \begin{tiny}$U^{(3)}_{1,3}$\end{tiny}  & \begin{tiny}$U_{1,4}$ \end{tiny}&\begin{tiny}$U_{2,2}$  \end{tiny} &\begin{tiny}$U^{(1)}_{2,3}$  \end{tiny}  &\begin{tiny}$U^{(2)}_{2,3}$  \end{tiny} &\begin{tiny}$U^{(3)}_{2,3}$   \end{tiny} & \begin{tiny}$U^{(1)}_{2,4}$\end{tiny} & \begin{tiny}$U^{(2)}_{2,4}$\end{tiny} & \begin{tiny}$U^{(1)}_{2,5}$\end{tiny} & \begin{tiny}$U^{(2)}_{2,5}$\end{tiny}& \begin{tiny}$U_{2,6}$\end{tiny}& \begin{tiny}$U_{3,3}$\end{tiny}& \begin{tiny}$U_{3,4}$ \end{tiny}& \begin{tiny}$U^{(1)}_{3,5}$\end{tiny} & \begin{tiny}$U^{(2)}_{3,5}$\end{tiny} & \begin{tiny}$U_{3,6}$\end{tiny}& \begin{tiny}$U^{(1)}_{3,6}$\end{tiny}& \begin{tiny}$U^{(2)}_{3,6}$\end{tiny} & \begin{tiny}$U_{3,7}$\end{tiny} & \begin{tiny}$U_{4,6}$\end{tiny} & \begin{tiny}$U_{4,7}$\end{tiny}\\ 
\hline
\begin{tiny}$U_{0,1}$\end{tiny} &$1$ & $1$ & $1$  & $0$ & $1$  & $1$  & $1$   & $1$  & $1$  & $1$  & $2$ & $1$   & $1$  & $1$  &   $1$  &  $2$ & $2$ & $2$ & $2$ & $2$ & $1$  & $2$ & $2$ & $2$ &   $2$  & $2$ & $2$ & $3$ &   $2$ &   $3$  \\ 
\hline
\begin{tiny}$U_{0,2}$ \end{tiny}&$1$ & $2$ & $2$  & $0$ & $1$  & $2$  & $2$   & $2$  & $2$  & $2$  & $3$ & $2$   & $2$   & $2$  &  $2$ &  $3$ &  $3$& $4$ & $4$ & $4$ & $2$ &$3$  & $4$ &$4$ &   $4$  & $4$ & $4$ & $5$ &   $4$&   $5$  \\ 
\hline
\begin{tiny}$U_{1,0}$ \end{tiny}&$0$ & $0$ & $0$  & $1$ & $0$  & $0$  & $0$   & $0$  & $0$  & $0$  & $0$ & $0$   & $0$   & $0$  &  $0$ &  $0$ &  $0$& $0$ & $0$ & $0$ & $0$ &$0$  & $0$ &$0$ &   $0$  & $0$ & $0$ & $0$ &   $0$&   $0$  \\
\hline
\begin{tiny}$U_{0,3}$\end{tiny} &$1$ & $2$ & $3$ & $0$  & $1$  & $2$  & $2$   & $3$  & $3$  & $3$  & $4$ & $2$   & $3$   & $3$   & $3$  &  $4$ & $4$  & $5$ & $5$ & $6$ & $3$ &  $4$& $5$ & $5$ &  $6$ & $6$ & $6$ & $7$ &  $6$ &   $7$    \\ 
\hline
\begin{tiny}$U_{1,1}$\end{tiny} &$0$ & $0$ & $0$  & $1$ & $1$  & $0$  & $1$   & $0$  & $0$  & $1$  & $0$ & $1$   & $0$   & $1$  &  $1$  & $0$ & $1$ &  $0$ &  $1$ &  $0$ & $1$ &  $1$ &  $1$& $1$ & $1$ & $0$ &   $1$ &   $1$&  $1$ & $1$  \\ 
\hline
\begin{tiny}$U^{(1)}_{1,2}$\end{tiny} &$0$ & $0$ & $0$  & $1$ & $1$  & $1$  & $1$   & $0$  & $1$  & $1$  & $1$ & $2$   & $1$   &$1$  & $2$  & $1$ & $2$ & $1$ & $2$ & $1$ & $2$ &  $2$ &  $2$  & $2$ & $2$ & $1$ &   $2$ &   $2$&  $3$ & $3$  \\ 
\hline
\begin{tiny}$U^{(2)}_{1,2}$\end{tiny} &$1$ & $1$ & $1$ & $1$  & $1$  & $1$  & $2$   & $1$ & $1$  & $2$  & $2$ & $2$   & $1$   & $2$  &  $2$   & $2$ & $3$ &$2$  & $3$ & $2$ & $2$ &  $2$ & $3$  & $3$ & $3$ & $2$ &   $2$  &   $4$&  $3$ & $4$ \\ 
\hline
\begin{tiny}$U^{(1)}_{1,3}$\end{tiny} & $0$ &  $0$ &  $0$  & $1$ & $1$  & $1$  & $1$   & $1$  & $1$  &$1$  & $1$ & $2$   & $2$   & $2$   & $2$  & $2$ & $2$  & $2$ & $2$ &  $2$ &$3$   &  $3$ & $3$  & $3$ & $3$ & $3$& $3$ &   $3$& $4$  & $4$ \\ 
\hline
\begin{tiny}$U^{(2)}_{1,3}$\end{tiny} &$1$ & $1$ & $1$ & $1$  & $1$  & $1$  & $2$   & $1$  & $2$  &$2$  & $2$ & $2$   & $2$   & $2$   & $2$  & $2$ & $3$ & $3$ & $3$ &  $3$ & $3$ & $3$  & $3$   & $4$& $4$ & $3$ & $4$  &   $4$& $4$  & $5$  \\ 
\hline
\begin{tiny}$U^{(3)}_{1,3}$\end{tiny} &$1$ & $2$ &$2$  &$1$  & $1$  & $2$  &$2$  & $2$  & $2$  & $3$  & $3$ & $2$   & $2$   & $2$   & $3$  & $3$ &$4$ & $4$ & $4$ &  $4$  & $3$  & $4$ & $4$  & $5$ & $5$ & $4$ &  $5$  &   $6$& $5$ & $6$    \\ 
\hline
\begin{tiny}$U_{1,4}$\end{tiny} &$1$ & $2$ & $2$ & $1$   & $2$  & $2$  & $3$  & $2$  & $3$  & $3$  &$4$ & $2$   & $3$   & $3$   & $4$  & $4$ & $5$ &$5$ & $6$ &$6$  & $4$&  $5$ &  $5$  & $6$ & $5$ & $5$ &   $6$&   $7$&  $5$ & $8$   \\ 
\hline
\begin{tiny}$U_{2,2}$\end{tiny} &$0$& $0$ & $0$ & $2$   &$1$  & $0$  & $1$  & $0$  & $0$  & $1$  & $0$ & $2$   & $0$   & $1$  &  $2$  & $0$& $2$ & $0$ & $1$ & $0$ &$2$ &  $2$ &  $1$   &$2$ & $1$ & $0$ &  $2$&   $1$&  $2$ & $2$  \\ 
\hline
\begin{tiny}$U^{(1)}_{2,3}$\end{tiny} &$0$ & $0$ & $0$  & $2$ & $1$  &$0$  & $1$   & $0$  & $0$  & $1$ & $0$ & $2$   & $1$   & $1$   & $2$    & $0$ & $2$  & $0$  &  $1$ &  $0$ & $3$   &  $2$ & $1$   & $2$ &  $1$ &  $0$ &   $2$&    $1$&   $2$ & $2$  \\ 
\hline
\begin{tiny}$U^{(2)}_{2,3}$\end{tiny} &$0$ & $0$ & $0$ & $2$  & $1$  &$1$  & $1$   & $0$  & $1$  & $1$  & $1$ & $2$   & $1$   & $2$  & $2$   & $1$ & $2$  &  $1$ &  $2$ &  $1$ &  $3$ &  $3$  & $2$   &$3$ &  $2$ &  $1$ &   $3$ &    $2$&   $3$ & $3$   \\ 
\hline
\begin{tiny}$U^{(3)}_{2,3}$\end{tiny} & $1$ &$1$ & $1$  & $2$& $1$  &$1$  & $2$   & $1$  & $1$  & $2$  & $2$ & $2$   & $1$   & $2$   &  $3$   & $2$ &  $3$ &  $2$& $3$  & $2$ & $3$  &  $3$ & $3$  & $4$ &  $3$ &  $2$ &    $4$ &    $4$&  $4$ & $4$    \\ 
\hline
\begin{tiny}$U^{(1)}_{2,4}$\end{tiny} &$0$ & $0$ & $0$  & $2$  & $2$   &$1$   & $2$    &$0$  & $1$  & $2$   & $1$  & $3$   & $2$  & $2$   & $3$   & $2$ & $3$ & $2$ &$3$ & $2$ & $4$&  $4$ & $3$  & $4$ & $3$ & $2$ &   $4$ &   $3$&  $4$ & $5$   \\ 
\hline
\begin{tiny}$U^{(2)}_{2,4}$\end{tiny} &$1$ & $2$ & $2$  & $2$  & $2$   & $2$   & $3$    & $2$   & $2$   &$3$   & $3$  & $3$   &$2$   & $3$   & $4$   & $3$  & $5$ & $4$ & $5$ & $4$ & $4$ &  $5$ & $5$   & $6$ & $5$ & $4$&  $6$  &   $6$&  $6$ &$7$  \\ 
\hline
\begin{tiny}$U^{(1)}_{2,5}$\end{tiny} &$1$ & $1$ & $1$  & $2$ & $2$  & $2$  & $3$   &  $1$ & $2$  & $3$  & $3$ & $4$   & $3$   & $3$  & $4$   & $3$ & $5$ & $4$ &$5$  & $4$ & $5$ & $5$&    $5$ &  $6$ &  $5$ &  $4$ &    $6$&    $6$&   $6$ &  $8$   \\ 
\hline
\begin{tiny}$U^{(2)}_{2,5}$\end{tiny} &$1$ & $2$ & $2$  &$2$ & $2$  & $2$  & $3$   & $2$ & $3$  & $4$  & $4$ & $4$   & $3$   & $4$  & $5$  & $4$  & $6$ & $5$ & $6$ & $3$ & $5$ &   $6$ &    $6$  &  $7$ &  $7$ &  $5$ &    $7$&    $8$&   $6$ &  $9$   \\ 
\hline
\begin{tiny}$U_{2,6}$\end{tiny} &$1$ & $2$ & $2$   & $2$  & $2$  & $2$  & $3$ & $2$  & $3$  & $4$  & $4$ & $4$   & $4$   & $4$  &  $5$  &  $4$ &  $6$ &  $5$ & $6$ & $6$ & $6$ &  $6$ &    $6$ & $8$ & $7$ & $6$ &   $8$&   $8$&  $8$ & $9$   \\ 
\hline
\begin{tiny}$U_{3,3}$\end{tiny} &$0$ & $0$& $0$  & $3$ & $1$  & $0$ & $1$   & $0$  & $0$  & $1$  &$0$ & $2$   &$0$   & $1$   &  $2$     & $0$ &  $2$& $0$ &$1$ & $0$ & $3$ &   $2$& $1$   & $2$ & $1$ & $0$ &   $2$  &   $1$&  $2$ & $2$   \\ 
\hline
\begin{tiny}$U_{3,4}$\end{tiny} &$0$ &$0$ & $0$ & $3$ & $2$  & $1$  & $2$   & $0$  &$1$  & $2$  &$1$ & $3$   & $2$   & $2$ &  $4$   & $1$  & $3$ & $1$  & $3$ & $1$  & $4$ & $4$  &   $3$ & $4$ & $3$ &$1$ &   $4$&   $3$&  $4$ & $4$ \\ 
\hline
\begin{tiny}$U^{(1)}_{3,5}$\end{tiny} &$0$ & $0$ & $0$ & $3$  & $2$  & $1$  & $2$   & $0$ & $1$  & $2$  & $1$ & $4$   & $2$   & $3$   & $4$  &$2$ & $4$ &$2$ & $3$ & $2$  & $5$ &  $5$ & $4$  & $5$ & $4$ & $2$ &   $5$ &  $4$&  $5$ & $6$     \\ 
\hline
\begin{tiny}$U^{(2)}_{3,5}$\end{tiny} &$1$ & $1$ & $1$ & $3$ &$2$  & $2$  & $3$  & $1$  & $2$  &$3$  & $3$ & $4$   & $2$   & $3$  &  $4$   & $3$ & $5$  & $3$ & $5$ & $3$ & $5$ &  $5$ &  $5$   & $6$ & $5$ & $3$ &   $6$&   $6$&  $6$ & $7$   \\ 
\hline
\begin{tiny}$U_{3,6}$\end{tiny} &$1$ & $1$ & $1$ & $3$  & $2$  & $2$  & $3$  & $1$  & $2$  & $3$  & $3$ & $4$   & $3$   & $4$   & $5$   & $4$ & $5$ & $4$ & $5$ & $4$ &  $6$ &   $6$ &   $5$  &  $7$ &  $6$ &  $4$ &    $7$&    $6$&   $7$ &  $8$  \\ 
\hline
\begin{tiny}$U^{(1)}_{3,6}$\end{tiny} &$0$ & $0$ & $0$   & $3$  & $2$ & $1$  & $2$ & $0$  & $1$  & $2$  & $2$ & $4$  & $3$   & $3$  &  $4$   & $2$ & $4$ &$2$ & $3$ & $2$ & $6$&   $5$ &     $4$ &  $5$ &  $4$ &  $3$ &   $4$&    $5$&   $6$&  $6$  \\ 
\hline
\begin{tiny}$U^{(2)}_{3,6}$\end{tiny} &$1$ & $2$ & $2$   & $3$  & $2$  &$2$  & $3$ & $1$  & $2$  & $3$  & $3$ & $4$   & $3$   & $4$  &  $5$ & $5$& $6$ & $5$ & $6$ &$5$ & $6$&   $6$ &     $6$ & $7$ &  $7$ & $5$ &    $8$&    $8$&  $8$ &  $8$  \\ 
\hline
\begin{tiny}$U_{3,7}$\end{tiny} &$1$ & $2$ &$2$  & $3$& $3$  & $3$  & $4$   & $2$  & $3$  & $4$  &  $4$ &  $5$   &  $4$   &  $5$  &  $6$   & $5$ &$7$ & $5$ & $7$ &$6$  &$7$  &  $8$ &   $7$  &$9$  &$8$  & $6$  &   $9$ &   $9$ &  $8$  & $11$    \\ 
\hline
\begin{tiny}$U_{4,6}$\end{tiny} & $0$ & $0$ &$0$   & $4$  & $2$  & $2$  & $2$ & $0$  & $1$  &$2$ & $2$& $4$   &$2$   & $3$  &  $4$ & $2$  & $4$  &$2$  & $3$  & $3$  & $6$  &  $5$  &   $4$  &$5$ & $4$  & $2$  &   $5$ &   $4$ &  $3$ & $6$  \\ 
\hline
\begin{tiny}$U_{4,7}$\end{tiny} &$1$ &$1$ &$1$   & $4$ & $3$  &$2$  & $3$& $1$  & $2$  & $4$  & $3$ & $5$   & $3$  & $4$ &  $6$    & $3$ & $6$ & $4$ & $6$ & $4$ & $7$ &  $8$ &    $6$ & $8$ & $6$ & $4$ &   $8$&  $8$&  $7$ & $9$ \\ 
\hline
\end{longtable} \caption{Dimensions of homomorphism spaces}\label{fig2}
\end{figure}\end{landscape}

\section{A finiteness criterion}\label{fincrit}
We consider the $P$-action on $\N_n^{(x)}$ and prove a criterion as to whether the action admits finitely many or infinitely many orbits. 
\begin{theorem}\label{classfinpar}
 There are only finitely many $P$-orbits in $\N_n^{(x)}$ if and only if $x\leq 2$, or $P$ is maximal and $x=3$.
\end{theorem}
\begin{proof}
If $x=2$, our considerations in Section \ref{x2} yield finiteness for every parabolic subgroup $P$; if $x=3$ and $P$ is maximal, then our considerations in Section \ref{degree3} yield the claim.\\[1ex]
Now let $P$ be a non-maximal parabolic subgroup and let $x\geq 3$. The action of $P$ on $\N_n^{(x)}$ admits infinitely many orbits, because 
\[(D_x(\lambda))_{i,j}=\left\lbrace\begin{array}{ll}
\lambda, & \textrm{if}~i=n~\textrm{and}~j=1; \\ 
1, & \textrm{if}~(1\leq i<n~\textrm{and}~j=1)~ \textrm{or}~(i=n~\textrm{and}~1\leq j<n);\\ 
0, & \textrm{otherwise}. 
                                   \end{array}\right.\]
yields a $1$-parameter family of pairwise non-$P$-conjugate matrices for $\lambda\in K^*$.\\[1ex]
If $P$ is a maximal parabolic subgroup of block sizes $(x,y)$, then the action of $P$ on $\N_n^{(4)}$  admits infinitely many orbits:
\begin{enumerate}
 \item If $x=s+2\geq 2$ and $y=t+2\geq 2$ for $s,t\leq 0$, then the matrices
\[(E^s(n,\lambda))_{i,j}:=\left\lbrace\begin{array}{ll} 
(E(\lambda))_{i-s,j-s}, & \textrm{if}~s+1\leq i,j\leq s+4; \\ 
0, & \textrm{otherwise}.
                                   \end{array}\right.\]
where \[E(\lambda)\coloneqq\left( \begin{array}{cccc}
0 & 0 & 0 & 0 \\ 
1 & 0 & 0 & 0 \\ 
1 & 1 & 0 & 0 \\ 
\lambda & 1 & 1 & 0
    \end{array}\right)\]
for $\lambda\in K^*$, induce a $1$-parameter family of pairwise non-$P$-conjugate matrices.
\item If (without loss of generality) $x=1$ and $y=n-1$, then for $\lambda\in K^*$, the matrices 
\[(F(n,\lambda))_{i,j}=\left\lbrace\begin{array}{ll}
(F(\lambda))_{i,j}, & \textrm{if}~1\leq i,j\leq 4; \\ 
0, & \textrm{otherwise}.
                                   \end{array}\right.\]
where \[F(\lambda)\coloneqq\left( \begin{array}{cccc}
1 & 1 & 0 & 0 \\ 
-1 & -1 & 0 & 0 \\ 
\lambda-1 &  \lambda & -1 & 1 \\ 
\lambda &\lambda-1 & -1 & 1
    \end{array}\right)\]
induce a $1$-parameter family of pairwise non-$P$-conjugate matrices.\qedhere
\end{enumerate}
 \end{proof}

\section{A wildness criterion}
Let us fix $p>1$. Theorem \ref{fincrit} shows that the algebra $K\Q_p/I_x$ is of finite representation type if and only if $x\in\{1,2\}$, or $p=2$ and $x=3$.
In this section, it will be shown that each remaining algebra is of wild representation type.
\begin{proposition}\label{reptype}
 The algebra $K\Q_p/I_x$ is of wild representation type if and only if it is not of finite representation type.
\end{proposition}
\begin{proof} If $K\Q_p/I_x$ is not of finite representation type, then either $x=3$ and $p>2$, or $x\geq 4$.\\[1ex]
 If $x=3$ and $p>2$, then the covering quiver of $K\Q_p/I_x$ at the vertex $p$ contains the subquiver
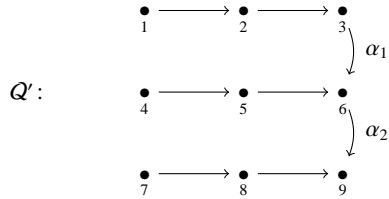
\begin{center}{\small\begin{tikzpicture}
\matrix (m) [matrix of math nodes, row sep=2em,
column sep=3em, text height=1.5ex, text depth=0.2ex]
{ & \underset{1}\bullet& \underset{2}\bullet  & \underset{3}\bullet \\ 
\Q'\colon & \underset{4}\bullet& \underset{5}\bullet  & \underset{6}\bullet \\
 & \underset{7}\bullet& \underset{8}\bullet  & \underset{9}\bullet \\};
\path[->]
(m-1-2) edge  (m-1-3)
(m-2-2) edge   (m-2-3)
(m-3-2) edge  (m-3-3)
(m-1-3) edge  (m-1-4)
(m-2-3) edge   (m-2-4)
(m-3-3) edge  (m-3-4)
(m-1-4) edge[bend left=20] node[right=0.05cm] {$\alpha_{1}$} (m-2-4)
(m-2-4) edge[bend left=20] node[right=0.05cm] {$\alpha_2$} (m-3-4);\end{tikzpicture}}\end{center}
without any relations. \\[1ex]
If $x\geq 4$, then the covering quiver of $K\Q_p/I_x$ at the vertex $p$ contains the subquiver
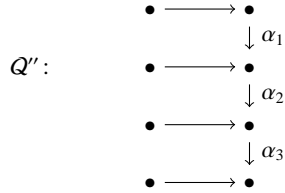
\begin{center}{\small\begin{tikzpicture}\matrix (m) [matrix of math nodes, row sep=1em,
column sep=3em, text height=1.5ex, text depth=0.2ex]
{ & \bullet  & \bullet \\ 
\Q''\colon  & \bullet  & \bullet \\
& \bullet  & \bullet \\
 & \bullet  & \bullet \\};
\path[->]
(m-1-2) edge  (m-1-3)
(m-2-2) edge   (m-2-3)
(m-3-2) edge  (m-3-3)
(m-4-2) edge  (m-4-3)
(m-1-3) edge node[right=0.05cm] {$\alpha_{1}$} (m-2-3)
(m-2-3) edge node[right=0.05cm] {$\alpha_2$} (m-3-3)
(m-3-3) edge node[right=0.05cm] {$\alpha_{3}$} (m-4-3);\end{tikzpicture}}\end{center}
without any relations.\\[1ex]
These subquivers are not quivers of extended Dynkin types, therefore, the algebra $K\Q_p/I_x$ is of wild representation type.
\end{proof}
We have shown that $K\Q_p/I_x$ is never of infinite tame representation type.\\[1ex]
Note that we cannot conclude that each parabolic action admits $2$-parameter families of non-conjugate matrices. It is possible that certain parabolic actions admit at most $1$-parameter families of pairwise non-conjugate matrices - one example is the Borel-action on the nilpotent cone for $n=3$.
It is natural to try to exhibit a $2$-parameter family of pairwise non-conjugate matrices for at least one parabolic action corresponding to $K\Q_p/I_x$, though. \\[1ex]
By following a method for constructing indecomposable modules T. Weist describes in \cite{Wei1}, we are able to find such:
Let $U$ and $U'$ be two indecomposable representations of a finite-dimensional path algebra $\A=K\Q$, such that $\dimv U$ and $\dimv U'$ are real roots and such that the root $\dimv U+\dimv U'$ is an imaginary root of $\Q$. Assume that $[U',U]=0=[U,U']^1$ and $[U',U]^1=3$, then the representatives $X$ of the middleterms of the classes of extensions 
$$[0\rightarrow U\rightarrow X\rightarrow U'\rightarrow 0]$$ 
yield a $2$-parameter family of pairwise non-isomorphic indecomposable $\A$-representations.\\[1ex]
We consider the two cases that come up in the proof of Proposition \ref{reptype} and find the following matrices by making use of the above described method. The proof follows directly from the construction, though, but can be calculated straight forward as well. We describe the first case in detail, the second one is left to the reader.\\[1ex]
Let $P$ be the parabolic subgroup $P$ of block sizes $(3,4,3)$. 
\begin{proposition}
A $2$-parameter family of pairwise non-$P$-conjugate matrices in $\N^{(3)}$ is induced by the matrices $$% use packages: array
N_{\lambda,\mu}=\left( \begin{array}{cccccccccc}
0 & 0 & 0 &0 & 0&0 &0&0&0&0\\ 
0 & 0 & 0 &0 & 0&0 &0&1&0&0\\ 
0 & 1 & 0 &0 & 0&0 &0&0&0&0\\ 
0 & 0 & 0 &0 & 0 &0 &0&0&0&0\\ 
0 & 0 & 0 &0 & 0&0 &0&-1&0&0\\ 
1 & 0 & 0 &0 & 0&0 &0&0&0&0\\ 
0 & 1 & 0 &0 & 1&-\mu &0&0&1&0\\  
0 & 0 & 0 &0 & 0&0 &0&0&0&0\\  
\lambda & 0 & 0 &1 & 0&0 &0&0&0&0\\ 
0 & -1 & 0 &0 & 0&1 &0&0&0&0\end{array}\right) $$
for $\lambda,\mu\in K^*$.
\end{proposition}
\begin{proof}
Let us consider the quiver $Q'$ of Proposition \ref{reptype} and the indecomposable $K\Q$-representations $U_{\ef}$ and $U_{\eef}$, where $\dimv U_{\ef}=:\ef$ and $\dimv U_{\eef}=:\eef$:
\begin{center}\begin{tikzpicture}[descr/.style={fill=white,inner sep=1.5pt}]
\matrix (m) [matrix of math nodes, row sep=2em,
column sep=3em, text height=1.5ex, text depth=0.2ex]
{ &  0 & K  & K & &  K & K  & K^2\\ 
U_{\ef}\colon & 0 & K  & K^2&  U_{\eef}\colon & K & K^2  & K^2 \\
 & 0 & K  & K &  & K & K  & K^2 \\};
\path[->]
(m-1-2) edge node[descr] {\begin{scriptsize}$0$\end{scriptsize}} (m-1-3)
(m-2-2) edge node[descr] {\begin{scriptsize}$0$\end{scriptsize}} (m-2-3)
(m-3-2) edge node[descr] {\begin{scriptsize}$0$\end{scriptsize}} (m-3-3)
(m-1-3) edge node[descr] {\begin{scriptsize}$\id$\end{scriptsize}} (m-1-4)
(m-2-3) edge node[descr] {\begin{scriptsize}$e_2$\end{scriptsize}} (m-2-4)
(m-3-3) edge node[descr] {\begin{scriptsize}$\id$\end{scriptsize}} (m-3-4)
(m-1-4.south) edge node[descr] {\begin{scriptsize}$e_1$\end{scriptsize}} (m-2-4)
(m-2-4) edge node[descr] {\begin{scriptsize}$(e_1+e_2)^t$\end{scriptsize}} (m-3-4)
(m-1-6)  edge node[descr] {\begin{scriptsize}$\id$\end{scriptsize}} (m-1-7)
(m-2-6)  edge node[descr] {\begin{scriptsize}$e_1$\end{scriptsize}}  (m-2-7)
(m-3-6)  edge node[descr] {\begin{scriptsize}$\id$\end{scriptsize}} (m-3-7)
(m-1-7)  edge node[descr] {\begin{scriptsize}$e_1+e_2$\end{scriptsize}} (m-1-8)
(m-2-7)  edge node[descr] {\begin{scriptsize}$\id$\end{scriptsize}}  (m-2-8)
(m-3-7)  edge node[descr] {\begin{scriptsize}$e_2$\end{scriptsize}} (m-3-8)
(m-1-8.south) edge node[descr] {\begin{scriptsize}$\id$\end{scriptsize}} (m-2-8)
(m-2-8) edge node[descr] {\begin{scriptsize}$\id$\end{scriptsize}} (m-3-8);\end{tikzpicture}\end{center}
Then $\ef$ and $\eef$ are positive real roots; their sum is an imaginary root $\df=(1,2,3,1,3,4,1,2,3)$ (that is, its Tits form is negative) 
which fulfills the assumptions of the above mentioned construction algorithm.
Since $[\U_{\eef},\U_{\ef}]^1=3$, we use the extensions to glue the two representations together in order to obtain the sought representations, here $\lambda,\mu\in K$:
\begin{center}\begin{tikzpicture}[descr/.style={fill=white,inner sep=2.5pt}]
\matrix (m) [matrix of math nodes, row sep=1em,
column sep=2em, text height=1.5ex, text depth=0.2ex]
{   0 & & K  & & K &\\ 
 & K & & K & & K^2 \\[2ex]
 0 & & K & & K^2&  \\
& K&  &K^2&& K^2 \\[2ex]
 0& & K & & K &   \\
&K  &&K &&K^2\\ };
\path[->]
(m-1-1) edge node[above=0.015cm] {\begin{scriptsize}$0$\end{scriptsize}} (m-1-3)
(m-3-1) edge node[above=0.015cm] {\begin{scriptsize}$0$\end{scriptsize}} (m-3-3)
(m-5-1) edge node[above=0.015cm] {\begin{scriptsize}$0$\end{scriptsize}} (m-5-3)
(m-1-3) edge node[above=0.015cm] {\begin{scriptsize}$\id$\end{scriptsize}} (m-1-5)
(m-3-3) edge node[above=0.015cm] {\begin{scriptsize}$e_2$\end{scriptsize}} (m-3-5)
(m-5-3) edge node[above=0.015cm] {\begin{scriptsize}$\id$\end{scriptsize}} (m-5-5)
(m-1-5) edge node[descr,below=0.2cm] {\begin{scriptsize}$e_1$\end{scriptsize}} (m-3-5)
(m-3-5) edge node[descr,below=0.2cm] {\begin{scriptsize}$(e_1+e_2)^t$\end{scriptsize}} (m-5-5)
(m-2-2)  edge node[above=0.015cm] {\begin{scriptsize}$\id$\end{scriptsize}} (m-2-4)
(m-4-2)  edge node[above=0.015cm] {\begin{scriptsize}$e_1$\end{scriptsize}}  (m-4-4)
(m-6-2)  edge node[above=0.015cm] {\begin{scriptsize}$\id$\end{scriptsize}} (m-6-4)
(m-2-4)  edge node[above=0.17cm,left=0.52cm] {\begin{scriptsize}$e_2$\end{scriptsize}} (m-2-6)
(m-4-4)  edge node[above=0.17cm,left=0.52cm] {\begin{scriptsize}$\id$\end{scriptsize}}  (m-4-6)
(m-6-4)  edge node[above=0.14cm,left=0.08cm] {\begin{scriptsize}$e_1+e_2$                      \end{scriptsize}} (m-6-6)
(m-2-6) edge node[descr] {\begin{scriptsize}$\id$\end{scriptsize}} (m-4-6)
(m-4-6) edge node[descr] {\begin{scriptsize}$\id$\end{scriptsize}} (m-6-6)
(m-2-2) edge[lblue,line width=0.03cm] node[lblue,line width=0.04cm,descr] {\begin{scriptsize}$\lambda$\end{scriptsize}} (m-1-3)
(m-4-2) edge[lblue,line width=0.03cm] node[lblue,line width=0.04cm,descr] {\begin{scriptsize}$\id$\end{scriptsize}} (m-3-3)
(m-6-2) edge[lblue,line width=0.03cm] node[lblue,line width=0.04cm,descr] {\begin{scriptsize}$\mu$\end{scriptsize}} (m-5-3);
\end{tikzpicture}\end{center}
We obtain a representation
\begin{center}\begin{tikzpicture}[descr/.style={fill=white,inner sep=2.5pt}]
\matrix (m) [matrix of math nodes, row sep=1em,
column sep=3em, text height=1.5ex, text depth=0.2ex]
{   K^3  &K^7 & K^{10} \\
};
\path[->]
(m-1-1) edge node[above=0.015cm] {$a_{\lambda,\mu}$} (m-1-2)
(m-1-2) edge node[above=0.015cm] {$b$} (m-1-3)
(m-1-3) edge[loop right] node {$A$} (m-1-3);
\end{tikzpicture}\end{center}
For fixed parameters $\lambda,\mu\in K^*$, this representation  is isomorphic to a unique representation of the form
\begin{center}\begin{tikzpicture}[descr/.style={fill=white,inner sep=2.5pt}]
\matrix (m) [matrix of math nodes, row sep=1em,
column sep=3em, text height=1.5ex, text depth=0.2ex]
{   K^3  &K^7 & K^{10} \\ 
};
\path[->]
(m-1-1) edge node[above=0.015cm] {$e_{3,7}$} (m-1-2)
(m-1-2) edge node[above=0.015cm] {$e_{7,10}$} (m-1-3)
(m-1-3) edge[loop right] node {$N_{\lambda,\mu}$} (m-1-3);
\end{tikzpicture}.\qedhere\end{center}
\end{proof}

\begin{proposition}
Let $P$ be the parabolic subgroup $P$ of block sizes $(5,5)$. A $2$-parameter family of pairwise non-$P$-conjugate matrices in $\N^{(3)}$ is induced by the following matrices; here $\lambda,\mu\in K^*$: $$% use packages: array
N_{\lambda,\mu}=\left( \begin{array}{cccccccccc}
0 & 0 & 0 &0 & 0&0 &0&0&0&0\\ 
1& 0 & 0 &0 & 0&0 &0&0&0&0\\ 
0 & 0 & 0 &0 & 0&0 &0&0&0&0\\ 
0 & 0 & 0 &0 & 0&0 &0&0&0&0\\ 
0 & 0 & 0 &1 & 0&0 &0&0&1&0\\ 
0 & 0 & 0 &0 & 0&0 &0&0&0&0\\ 
\lambda & 0 & 0 &0 & 0&1 &0&0&0&0\\  
0 & 0 & 1 &0 & 0&0 &1&0&0&0\\  
0 & 1 & 0 &0 & 0&0 &0&0&0&0\\ 
0 & 0 & 0 &1-\mu & 0&0 &0&1&-\mu&0\end{array}\right) $$
\end{proposition}

\end{document}